\documentclass[hidelinks,onefignum,onetabnum]{siamart220329}

\usepackage[T1]{fontenc}
\usepackage{amsmath,bm}
\usepackage{amssymb,mathtools,esint}
\usepackage[english]{babel}

\usepackage[textsize=tiny, textwidth=2cm]{todonotes}
\usepackage{csquotes}
\usepackage{graphicx}
\usepackage{enumitem}
\usepackage{booktabs} 			
\usepackage{multirow} 			
\usepackage{parskip} 			
\usepackage{cancel}

\newcommand{\dom}{\Omega}
\renewcommand{\i}{\imath}

\newcommand{\Ecut}{{\rm E}_{\rm cut}}
\newcommand{\Ecutref}{{\rm E}_{\rm cut,ref}}
\newcommand{\Ha}{{\rm Ha}}

\newcommand{\tr}{{\rm Tr}}
\newcommand{\per}{{\rm per}}
\newcommand{\Ran}{{\rm Ran}}
\newcommand{\Ker}{{\rm Ker}}
\newcommand{\loc}{{\rm loc}}
\newcommand{\Res}{{\rm Res}}

\newcommand{\trvo}{\underline{\rm Tr}}
\newcommand{\Evo}{\underline{E}}

\newcommand{\norm}[1]{\lVert #1\rVert}

\newcommand{\N}{\mathbb{N}}
\newcommand{\Z}{\mathbb{Z}}

\newcommand{\R}{\mathbb{R}}
\newcommand{\C}{\mathbb{C}}

\newcommand{\Ne}{{N_{\rm el}}}
\newcommand{\Ec}{{\Ecut}}

\newcommand{\cD}{\mathcal{D}}
\newcommand{\cE}{\mathcal{E}}

\newcommand{\cG}{\mathfrak{G}}
\newcommand{\cH}{\mathcal{H}}

\newcommand{\cM}{\mathcal{M}}
\newcommand{\cMvo}{\underline{\mathcal{M}}}

\newcommand{\cR}{\mathcal{R}}
\newcommand{\cZ}{\mathcal{Z}}
\newcommand{\cB}{\mathcal{B}}
\newcommand{\cS}{\mathcal{S}}

\newcommand{\cV}{\mathcal{V}}

\newcommand{\cW}{\mathcal{W}}

\newcommand{\fD}{\mathfrak{D}}
\newcommand{\fQ}{\mathfrak{Q}}

\newcommand{\bfG}{{\bm G}}

\newcommand{\bfx}{{\bm x}}

\makeatletter
\newcommand{\pseudoeqref}[1]{\textup{\tagform@{#1}}}
\makeatother

\newsiamremark{assumption}{Assumption}
\newsiamremark{remark}{Remark}

\newcommand{\etaene}[1]{\eta_{{}#1}}

\newcommand{\err}{\texttt{err}}

\title{Fully guaranteed and computable error bounds on the energy for periodic Kohn--Sham equations with convex density functionals~\thanks{This publication is part of a project that has received funding from the European
Research Council (ERC) under the European Union’s Horizon 2020 Research and Innovation Programme – Grant Agreement EMC2 No. 810367. GD
was supported by the French ‘Investissements d’Avenir’ program, project Agence
Nationale de la Recherche (ISITE-BFC) (contract ANR-15-IDEX-0003). GD was also supported by the Ecole des Ponts-ParisTech and region Bourgogne Franche-Comt\'e.
RL and BS acknowledge support from the Deutsche Forschungsgemeinschaft (DFG, German Research Foundation) under project 516782692. We thank the Deutsche Forschungsgemeinschaft (DFG, German Research Foundation) for supporting this work by funding - EXC2075 – 390740016 under Germany’s Excellence Strategy. We acknowledge the support by the Stuttgart Center for Simulation Science (SimTech).
}}

\author{Andrea Bordignon~\footnotemark[1]
\and{Eric Canc\`es~\footnotemark[1]}
\and{Genevi\`eve Dusson~\footnotemark[2]}
\and{Gaspard Kemlin~\footnotemark[3]}
\and{Rafael Antonio Lainez Reyes~\footnotemark[4]}
\and{Benjamin Stamm~\footnotemark[4]}}

\footnotetext[1]{Universit\'e Paris Est, CERMICS, Ecole des Ponts and
INRIA, 6 \& 8 Av. Pascal, 77455 Marne-la-Vall\'ee, France, \texttt{bordignon@cermics.enpc.fr}, \texttt{cances@cermics.enpc.fr}.}

\footnotetext[2]{Universit\'e de Franche-Comt\'e, CNRS, LmB, 25000 Besan\c con, France, \texttt{genevieve.dusson@math.cnrs.fr}.}

\footnotetext[3]{LAMFA, UMR CNRS 7352, Universit\'e de Picardie Jules Verne, 80039 Amiens, France, \texttt{gaspard.kemlin@u-picardie.fr}.}

\footnotetext[4]{IANS-NMH, University of Stuttgart, 70569 Stuttgart, Germany, \texttt{rafael-antonio.lainez-reyes@ians.uni-stuttgart.de}, \texttt{benjamin.stamm@ians.uni-stuttgart.de}.}

\begin{document}

\maketitle

\begin{abstract}
In this article, we derive fully guaranteed error bounds for the energy of convex non-linear mean-field models. These results apply in particular to Kohn--Sham equations with convex density functionals, which includes the reduced Hartree--Fock (rHF) model, as well as the Kohn--Sham model with exact exchange-density functional (which is unfortunately not explicit and therefore not usable in practice).
We then decompose the obtained bounds into two parts, one depending on the chosen discretization and one depending on the number of iterations performed in the self-consistent algorithm used to solve the nonlinear eigenvalue problem, paving the way for adaptive refinement strategies.
The accuracy of the bounds is demonstrated on a series of test cases, including a Silicon crystal and an Hydrogen Fluoride molecule simulated with the rHF model and discretized with planewaves. We also show that, although not anymore guaranteed, the error bounds remain very accurate for a Silicon crystal simulated with the Kohn--Sham model using nonconvex exchange-correlation functionals of practical interest.
\end{abstract}

\section{Introduction}
Ab initio simulations within the Born--Oppenheimer approximation~\cite{Born1927-ap} are performed routinely for simulating molecular and materials systems in several fields including condensed matter physics, chemistry and materials science.
The fact that the equations at stake do not depend on empirical parameters except a few fundamental constants of physics make them very appealing for systematic and accurate simulations.
A typical problem in this field, on which we will focus in this article, is the problem of finding the electronic ground state of the considered system, that is the state of lowest energy for the electrons, the nuclei (considered as point-like particles) being fixed at some given positions.
Among many models used in ab initio simulations, Density Functional theory (DFT) and especially Kohn--Sham models~\cite{Kohn1965} are the most popular ones, as they offer a good compromise between accuracy and computational cost. From a mathematical perspective, the Kohn--Sham models consist of a partial differential equation in the form of a nonlinear eigenvalue problem, as the corresponding differential operator, that needs to be diagonalized, depends on the eigenfunctions themselves.

\newpage

The numerical computations of solutions to the Kohn--Sham equations requires first to choose a discretization space, that is a finite-dimensional space on which the eigenvalue equations are projected and then solved.
Standard discretization methods for DFT equations include planewaves, linear combination of atomic orbitals (LCAO) or finite elements.
Once the discretization space is chosen, the nonlinear equations have to be solved, usually with a fixed-point algorithm called self-consistent field (SCF) algorithm, which is stopped when a given tolerance criterium is met.
Finally, in the SCF algorithm, a linear matrix eigenvalue problem is solved at each iteration, often resorting to an iterative linear algebra solver also stopped after a finite number of iterations.
Each of these steps leads to approximations, and it is important to estimate them for two reasons.
First, obtaining guaranteed bounds for the error between the exact and computed solutions allows to certify the accuracy of the computed solutions. Second, quantifying the size of the error coming from each approximation allows to optimize the parameters of the simulation in order to reduce the global computational cost necessary to reach a final total accuracy.

The purpose of this article is to derive fully guaranteed and computable \emph{a posteriori} energy bounds for DFT models with convex density functionals using a planewave discretization,
as well as to propose a decomposition of the total error between two contributions: one standing for the nonlinear solver error and the other for the discretization error. In this first contribution, we do not (yet) account for the error coming from the iterative eigenvalue solver of the successive (linear) matrix eigenvalue problems. Note that our results could be applied in theory to the Kohn--Sham model with exact exchange-density functional, since the exact density functional is convex (in the Vallone-Lieb version of DFT \cite{Valone1980,liebDensityFunctionalsCoulomb1983}). However, as the exact exchange-correlation functional has no known explicit expression, our bounds are useless in practice for the exact density functional. On the other hand, they can be applied to the reduced Hartree--Fock (rHF) model, i.e. to the Kohn--Sham model with exchange-correlation terms set to zero. In addition, they can be used in practice for usual density-functional approximations, namely LDA, GGA, etc. (see e.g. \cite{Toulouse2023} for a review on currently available approximate exchange-correlation functionals).
Although they are no longer guaranteed, numerical simulations indicate that our
bounds give a fairly good approximation of the actual error. Finally, we should mention that in periodic problems arising in condensed matter physics, the density is bounded away from zero. This physical property can be easily established with full mathematical rigor for Kohn--Sham LDA with PW92 exchange-correlation functional~\cite{perdewAccurateSimpleAnalytic1992}. In this framework, the energy functional is locally $\mathcal C^2$ in the neighborhood of any local minimizer. It is then reasonable to make the assumption that the Hessian is non-degenerate, hence that the energy functional is locally strongly convex in the neighborhood of the local minimizer under consideration, enabling extension to the present work to this case.

\medskip

Let us now put our work into perspective. The {\it a posteriori} error estimation of elliptic boundary value problems is already well developed, see e.g.~\cite{pragerApproximationsElasticityBased1947,braessEquilibratedResidualError2009,destuynderExplicitErrorBounds1999, ernPolynomialDegreeRobustPosterioriEstimates2015,ladevezeErrorEstimateProcedure1983,yanGradientRecoveryType2001}. Regarding eigenvalue problems,
the computation of guaranteed error
bounds were first proposed by  Kato \cite{katoUpperLowerBounds1949}, Forsythe
\cite{forsytheAsymptoticLowerBounds1954}, Weinberger
\cite{weinbergerUpperLowerBounds1956} or Bazley and Fox
\cite{bazleyLowerBoundsEigenvalues1961}, and more recently \emph{e.g.}\ in~\cite{cancesGuaranteedRobustPosteriori2017,cancesGuaranteedRobustPosteriori2018,carstensenGuaranteedLowerBounds2014,duranPosterioriErrorEstimates2003,
huLowerBoundsEigenvalues2014,huLowerUpperBound2014,larsonPosterioriPrioriError2000,
liuFrameworkVerifiedEigenvalue2015,luoComputingLowerUpper2012}, see also~\cite{nakaoNumericalVerificationMethods2019a} for a recent monograph on
the subject, and \cite{DUSSON2023112352} for a recent review. These error bounds are however only valid for single eigenvalues, while in practice,  degenerate or
near-degenerate eigenvalues often appear in electronic structure calculations. \emph{A posteriori} error estimates
for conforming approximations of eigenvalue clusters of second-order
self-adjoint elliptic operators with compact resolvent have more recently been proposed in
\cite{cancesGuaranteedPosterioriBounds2020a}, as well as in the context of electronic structure
calculations for linear eigenvalues equations in~\cite{herbstPosterioriErrorEstimation2020}.
When it comes to nonlinear eigenvalue equations as the ones of interest in this article, no guaranteed error bounds have been published to our knowledge, apart from the plane-wave discretization of a toy Gross--Pitaevskii equation in one-dimension~\cite{dussonPosterioriAnalysisNonlinear2017}. However, asymptotic error bounds can be found in \emph{e.g.}~\cite{Maday2003-eb} for the Hartree--Fock equations. For Kohn--Sham DFT, practical, not guaranteed, error bounds on the energy and the density matrix have been proposed for various discretization methods, notably planewave~\cite{Cances2016-vy,dussonPostprocessingPlanewaveApproximation2021}, and finite elements~\cite{Chen2014-qu,Lin2016-fn,Motamarri2012-fd,yangEigenfunctionBehaviorAdaptive2021,Zhang2017-ok}. Let us also mention the recent work \cite{cancesPracticalErrorBounds2022}, which includes practical bounds on quantities of interest other than the energy and the density matrix, such as the interatomic forces.

\medskip

The rest of the article is organized as follows. In Section~\ref{sec:2}, we present the model of interest together with the discretization of the equations and the self-consistent field (SCF) algorithm used to solve the discretized equations in practice. We then derive the proposed guaranteed bounds on the considered energy in Section~\ref{sec:3}.
Finally, we present numerical results on a set of test systems in Section~\ref{sec:4}, demonstrating the accuracy of the error bound on the energy for convex density functionals and showing that the bounds also provide very accurate error estimates in the case of nonconvex approximate density functionals used in practice.

\section{Problem setting, discretization, and practical resolution} \label{sec:2}

In this section we introduce the model of interest using an abstract framework, similar to ~\cite{cancesPracticalErrorBounds2022,cancesConvergenceAnalysisDirect2021a}. Our description, formulated using a \emph{density matrix formalism}, allows us to cover different models such as the (spinless) Kohn--Sham models,  the Hartree--Fock model, or the stationary Gross--Pitaevskii equation.

\subsection{Functional setting}

Let $\Omega\subset\R^{3}$ denote the unit cell of an arbitrary periodic lattice \(\cR\), with dual lattice \(\cR^*\).  For $p>0$, we denote by $L^p_{\per}(\Omega)$ the space of complex-valued, $p$-integrable $\cR$-periodic functions
\begin{equation}
    L^p_{\per}(\Omega)=\{u\in L^p_{\loc}(\R^3;\C): u\text{ is }\cR\text{-periodic}\}.
\end{equation}
In particular, the space $\cH \coloneqq L^2_{\per}(\Omega)$ with inner product $\langle \cdot, \cdot \rangle$ admits an orthonormal basis consisting of planewaves:
\begin{equation}
    e_\bfG:\bfx\in\R^3\mapsto |\Omega|^{-\frac{1}{2}}e^{\i\bfG\cdot\bfx},\;\bfG\in\cR^*.
\end{equation}
For $s\in\R$, the $\cR$-periodic Sobolev   space of order $s$ is then defined as
\begin{equation}
    H^s_\per(\Omega)\coloneqq\left\lbrace u(\bfx)=\sum_{\bfG\in\cR^*}\hat{u}_\bfG e_{\bfG}(\bfx):\ \sum_{\bfG\in\cR^*}\Big(1+\frac{|\bfG|^2}2\Big)^s|\hat{u}_\bfG|^2<\infty
    \right\rbrace,
\end{equation}
where $\hat u_\bfG$ represents the Fourier coefficient
\begin{equation}
    \hat{u}_\bfG = \langle e_\bfG,u\rangle = |\Omega|^{-\frac{1}{2}}\int_\Omega u(\bfx)e^{-\i\bfG\cdot \bfx}{\rm d}\bfx.
\end{equation}
Endowed with the inner product
\begin{equation}
    \langle u,v \rangle_{H^s_{\per}(\Omega)}\coloneqq\sum_{\bfG\in\cR^*}\Big(1+\frac{|\bfG|^2}2\Big)^s\hat{u}_\bfG^*\hat{v}_\bfG,\quad \forall u,v\in H_{\per}^s(\Omega),
\end{equation}
the space \(H^s_\per(\Omega)\) is a Hilbert space. We denote by $L^p_\per(\Omega;\R)$ and $H^s_\per(\Omega;\R)$ the spaces of real-valued functions in $L^p_\per(\Omega)$ and $H^s_\per(\Omega)$ respectively.

\subsection{Traces of operators and density matrices}

Let us now recall the definition of the trace of an operator. Suppose $A$ is a bounded, positive linear operator on $\cH$, then the \emph{trace} of $A$ is equal to
\begin{equation}
    \tr(A) \coloneqq \sum_{i\in\N} \langle \psi_i,A\psi_i\rangle \; \in \R_+ \cup \{+\infty\},
\end{equation}
where $(\psi_i)_{i\in\N}$ is any orthonormal basis of $\cH$. A bounded linear operator is \emph{trace class} if $\tr(|A|)<+\infty$ where $|A|:=(A^*A)^{1/2}$. For more details on trace-class operators and the underlying functional analysis setting, we refer to \cite{reedAnalysisOperators1978}, \cite[Section 2.2]{cancesGuaranteedPosterioriBounds2020a}, and references therein.

More generally, let us explain how to define $\tr(AB)$, in the case where (i) $A$ is a linear, self-adjoint, bounded from below, operator with domain $\fD(A)$ and form domain $\fQ(A)$, and (ii) $B$ is a finite-rank operator such that $\Ker(B)^{\perp}\subset\fQ(A), \Ran(B)\subset\fQ(A)$. First we write $B$ in canonical form, that is
\begin{align}
    B=\sum_{i=1}^{r}\sigma_i |\phi_i\rangle\langle\psi_i|,
\end{align}
where $r \in \N$ is the rank of $B$, $0<\sigma_r \le \cdots \le \sigma_1$ the singular values of $B$, $\phi_i\in \fQ(A)$, $\psi_i \in \fQ(A)$, and $( \phi_i)_{1 \le i \le r}$ and $(\psi_i)_{1 \le i \le r}$ are orthonormal with respect to the $L^2_\per(\Omega)$-inner product. Then we define
\begin{align}\label{def: weak trace}
\tr(AB)\coloneqq\sum_{i=1}^{r}\sigma_i\langle\psi_i,A\phi_i\rangle.
\end{align}
It is easy to check that this definition does not depend on the chosen canonical decomposition of $B$ and obviously coincides with the usual trace when $A$ is a bounded operator.

\medskip

In this work, we will focus on the set \(\cM\) of  finite-energy \emph{density matrices}, defined as the set of all rank-$\Ne$ orthogonal projectors on \(\cH=L^2_\per(\Omega)\) with range in $H^1_\per(\Omega)$, \textit{ i.e.}
\begin{equation}\label{eq:cM}
	\cM \coloneqq \Big\{ \gamma \in \cS(\cH),\;\gamma^2=\gamma,\;\tr(\gamma)=\Ne,\; \tr(-\Delta\gamma)<\infty \Big\},
\end{equation}
where \(\cS(\cH)\) is the set of all bounded self-adjoint operators over $\cH$, and $\Ne$ is a fixed integer, whose physical meaning will be clarified in the following section.
\\
For the sake of clarity, let us rewrite the term $\tr(-\Delta\gamma)$ in~\eqref{eq:cM} in a more explicit form.
Since $\gamma\in\cM$, there exists an $L^2_\per(\Omega)$-orthonormal basis $\Phi=(\phi_i)_{1\leq i \leq \Ne}$ of $\Ran(\gamma)~\subset~\fQ(-\Delta)~=~H^1_\per(\Omega)$ such that, using Dirac bra-ket notation, the projector $\gamma$ can be written as
\begin{equation}\label{eq: density matrix diag}
\gamma = \gamma_\Phi = \sum_{i=1}^\Ne |\phi_i\rangle\langle\phi_i|,
\end{equation}
and using \eqref{def: weak trace}, we can write
\begin{align}
    \tr(-\Delta\gamma)= \sum_{i=1}^\Ne\int_\Omega|\nabla\phi_i|^2<+\infty.
\end{align}

\subsection{Problem formulation}\label{sec:dm}

From now on the constant $\Ne$ appearing in definition \eqref{eq:cM} will represent the number of electron in the system. Now consider $\gamma\in\cM$, written as in~\eqref{eq: density matrix diag} for some $\Phi=(\phi_i)_{1\leq i \leq \Ne}$ of $\Ran(\gamma)$.
We can associate to $\gamma$ its \emph{kernel}, still denoted by $\gamma$
\begin{equation}
\gamma({\bfx},{\bfx'}) = \sum_{i=1}^\Ne \phi_i(\bfx)\phi_i^*(\bfx'),
\end{equation}
and define the electronic density $\rho_\gamma\in L^1_{\per}(\Omega;\R)$ as
\begin{equation}\label{eq:rho}
\rho_\gamma(\bfx) \coloneqq \gamma(\bfx,\bfx) = \sum_{i=1}^\Ne |\phi_i(\bfx)|^2 \eqqcolon \rho_\Phi(\bfx).
\end{equation}
We would like to remark that as a consequence of $\phi_i\in H^1_\per(\Omega)$ for $i=1,\dots,\Ne$, $\rho_\gamma$ actually belongs to $L^3_\per(\Omega;\R)$ hence in $L^{2}_\per(\Omega;\R)$ since $\Omega$ is bounded.  Finally we will consider a real-valued functional $E$ taking values over $\cM$, which we will refer to as the \emph{energy functional}. When looking for the ground state of the system, we seek the density matrix $\gamma_\star\in\cM$ that minimizes $E$:
\begin{equation} \label{eq:minE}
	\min \left\{ E(\gamma), \; \gamma \in \cM \right\}.
\end{equation}

In the models studied in this work (see Section~\ref{subsec:kohn_sham}), the energy $E$ can be expressed as the sum of a linear term and a nonlinear term:
\begin{equation}\label{eq:energy_full}
	E(\gamma) \coloneqq\tr(h \gamma) + F(\rho_\gamma).
\end{equation}
Here, $h$ is some bounded-from-below self-adjoint operator on $\cH$ with domain $H^2_\per(\Omega)$ and form domain $H^1_\per(\Omega)$, and $F$
is a function that explicitly depends on the electronic density $\rho_\gamma$ of the density matrix $\gamma\in \cM$.  We assume in the following that \(F:L^{2}_{\per}(\Omega;\R)\to\R\) is continuously differentiable. Under these conditions, problem~\eqref{eq:minE} admits a set of Euler--Lagrange equations, which we proceed to derive.

First, note that the assumptions on $F$ imply that for each $\gamma\in\cM$ there exists $V_{\rho_\gamma}\in L^2_\per(\Omega;\R)$ such that, for any $\widetilde{\gamma}\in\cM$:
 \begin{align}\label{eq: derivate F'(p) riesz form}
        \langle F'(\rho_{\gamma}),\rho_{\widetilde\gamma} \rangle=
        \int_\Omega V_{\rho_{\gamma}} \rho_{\widetilde\gamma}.
    \end{align}
With a slight abuse of notation, as in~\eqref{def: weak trace}, we denote
\begin{align}\label{eq:trV}
    \tr(V_{\rho_\gamma}\widetilde\gamma)\coloneqq \int_\Omega V_{\rho_{\gamma}} \rho_{\widetilde\gamma}.
\end{align}

 Now, if $\gamma\in\cM$ is a solution to problem~\eqref{eq:minE} the first-order variation of $E$ at $\gamma$ reads:
\begin{align*}
\forall\ \zeta\in T_\gamma\cM,
\quad \langle  E'(\gamma),\zeta\rangle = 0,
\end{align*}
where $T_\gamma\cM$, the tangent space at $\gamma$, is defined (see \cite{cances_convergence_2000}) by
\begin{equation}
T_\gamma\cM \coloneqq \lbrace \zeta \in \cS(\cH),\;\gamma\zeta+\zeta\gamma=\zeta,\; \tr(\zeta)=0,\; \tr(-\Delta\zeta)<\infty\rbrace.
\end{equation}
As before $\rho_\zeta\in L^{2}_\per(\Omega;\R)$, and as a consequence
\begin{equation}
\forall\ \zeta\in T_\gamma\cM,
\quad 0 = \langle  E'(\gamma),\zeta\rangle = \tr((h+V_{\rho_{\gamma}})\zeta)=\tr(H_{\rho_\gamma}\zeta),
\end{equation}
where, as a consequence of Kato-Rellich theorem \cite[Theorem X.12]{ReedFourierAnalysis1975} and the assumptions made on $h$, the operator $H_{\rho_\gamma} \coloneqq h+V_{\rho_{\gamma}}$ is self-adjoint on $\cH$ with domain $H^2_\per(\Omega)$ and form domain $H^1_\per(\Omega)$, and bounded from below.

Choosing $\zeta=|\phi\rangle\langle\psi|+|\psi\rangle\langle\phi|\in T_\gamma \cM$ with $\phi\in\operatorname{Ran}(\gamma)$ and $\psi\in\operatorname{Ker}(\gamma)$ (see \cite[Lemma 3]{cances_convergence_2000}) we finally obtain the Euler--Lagrange equations corresponding to problem \eqref{eq:minE}: find eigenvectors \(\Phi~=~(\phi_i)_{1\leq i\leq\Ne}~\in~(H^1_\per(\Omega))^{\Ne}\) of both $H_{\rho_\gamma}$ and $\gamma$, and eigenvalues \(\Lambda=(\lambda_i)_{1\leq i\leq\Ne}\in\R^{\Ne}\) of $H_{\rho_\gamma}$ such that
\begin{equation} \label{eq:eig_pb}
	\begin{cases}
		H_{\rho_\gamma} \phi_i = \lambda_i \phi_i, \quad  i=1,\ldots, \Ne, \\
  		\langle \phi_i,\phi_j \rangle = \delta_{ij}, \quad  i,j=1,\ldots,\Ne,\\
		\displaystyle \gamma = \sum_{i=1}^\Ne|\phi_i \rangle \langle \phi_i|.\\
	\end{cases}
\end{equation}
This is a \emph{nonlinear eigenvector problem}: the Hamiltonian $H_{\rho_\gamma}$ we seek to diagonalize depends on its own eigenvectors through the electronic density $\rho_\gamma$.

\begin{remark}[\emph{Aufbau} principe] We will assume in the following that for any minimizer~$\gamma_\star$ of~\eqref{eq:minE},
\begin{enumerate}
\item there is a positive gap between the $\Ne$-th and $(\Ne+1)$-st eigenvalues of the mean-field operator $H_{\rho_{\gamma_\star}}$ (counting multiplicities);
\item that $\gamma_\star$ is obtained from the eigenvectors associated to the lowest $\Ne$ eigenvalues of \eqref{eq:eig_pb} (\emph{Aufbau} principle).
\end{enumerate}
Note that for the reduced Hartree--Fock (rHF) model in $\Omega$ mentioned above, the function $F$ is in fact strictly convex, which implies that the minimizers of $E$ over the convex hull ${\rm CH}(\cM)$  of $\cM$ all share the same density $\rho_\star$. It can then be shown (see e.g. \cite{cancesNewApproachModeling2008,Solovej_1991}) that if $H_{\rho_\star}$ is gapped, then the minimizer of $E$ on ${\rm CH}(\cM)$ is unique, belongs to $\cM$, and is the spectral projector on the lowest  $\Ne$ eigenvalues of $H_{\rho_\star}$.

As a consequence, under reasonable assumptions~(see \emph{e.g.}\ \cite{cances_new_2008} for the reduced Hartree--Fock model), the solution \(\gamma\) to the minimization problem~\eqref{eq:minE} is unique, from where it follows that the eigenvectors, up to unitary transformations, are unique. This assumption is also crucial to use the results for eigenvalues clusters of self-adjoint operators presented in \cite{cancesGuaranteedPosterioriBounds2020a} in order to derive \emph{a posteriori} error estimates.
\end{remark}

\subsection{Kohn--Sham DFT equations}\label{subsec:kohn_sham}
The periodic Kohn--Sham equations are commonly used in numerical simulations of condensed matter systems. We now proceed to give a description of this model.
Let $\Phi \in (H^1_{\per}(\Omega))^\Ne$ be a set of $L^2$-orthonormal orbitals on $\Omega$. Denoting the electronic density $\rho_{\Phi}=\sum_{i=1}^\Ne|\phi_{i}|^2$, the Kohn--Sham energy is defined as
\begin{equation}\label{eq:KS-DFT-per}
    \cE(\Phi)\coloneqq\frac{1}{2}\sum_{i=1}^\Ne\int_\Omega|\nabla\phi_i|^2
    + \int_\Omega V\rho_\Phi
    +\frac{1}{2}\cD(\rho_\Phi,\rho_\Phi)+E_{\rm xc}(\rho_\Phi),
\end{equation}
where the first term represents the part of the kinetic energy corresponding to non-interacting electrons, the external potential $V \in L^2_\per(\Omega;\R)$ represents the interaction between the nuclei and the electrons, while the last two terms describe the interaction between the electrons. In particular, $E_{\rm xc}(\rho_\Phi)$ denotes the exchange-correlation energy while \(\mathcal{D}(\rho_1,\rho_2)\) stands for the Coulomb interaction-energy per unit cell
\begin{equation}
	\mathcal{D}(\rho_1,\rho_2) = \int_{\dom} V_{\rm H}[\rho_1](\bfx) \rho_2(\bfx) \, {\rm d}\bfx,
\end{equation}
where the Hartree potential \(V_{\rm H}[\rho]\) is the unique zero-mean solution in \(L^{2}_\per(\Omega;\R)\)  to the periodic Poisson equation
\begin{equation}
	\label{eq:VH}
	-\Delta V_{\rm H}[\rho] = 4\pi\left(\rho - \frac1{|\Omega|}\int_\Omega\rho\right),
\end{equation}
on $\Omega$.

In order to frame this minimization problem in the context of \eqref{eq:energy_full}, recall that \(\gamma_\Phi=\sum_{i=1}^\Ne |\phi_{i} \rangle \langle \phi_{i}|\) denotes the density matrix corresponding to the orbitals \(\Phi\), so we rewrite \eqref{eq:KS-DFT-per} in terms of the generalized trace~\eqref{def: weak trace} as
\begin{align}
    \cE(\Phi) = \tr\left(\Big(-\frac{1}{2}\Delta+V\Big)\gamma_\Phi\right)+\frac{1}{2}\cD(\rho_\Phi,\rho_\Phi)+E_{\rm xc}(\rho_\Phi) \eqqcolon E(\gamma_\Phi).
\end{align}
Thus, in the density matrix formalism, the Kohn--Sham minimization problem reads as in~\eqref{eq:minE}-\eqref{eq:energy_full} with the core Hamiltonian
\begin{equation}
    h=-\frac{1}{2}\Delta+V
\end{equation}
and density functional
\begin{equation}\label{eq:KS_nonlin}
    F(\rho)=\frac{1}{2}\cD(\rho,\rho)+E_{\rm xc}(\rho).
\end{equation}

Moreover, the corresponding Kohn--Sham Hamiltonian appearing in the Euler--Lagrange equations~\eqref{eq:eig_pb} takes the form
\begin{equation}\label{eq:KS_ham}
    H_{\rho_\gamma}=-\frac{1}{2}\Delta+V+V_{\rho_\gamma},
\end{equation}
where $V_{\rho} = V_{\rm H}[\rho] + V_{\rm xc}[\rho]$ and $V_{\rm xc}[\rho] = \frac{{\rm d} E_{\rm xc}(\rho)}{\rm d \rho}$ is the exchange-correlation potential.

\begin{remark}[Reduced Hartree--Fock model]
In Section~\ref{sec:4}, we will provide numerical simulations on the restricted Hartree--Fock model, which amounts to choosing $E_{\rm xc} \equiv 0$ in~\eqref{eq:KS-DFT-per}.

\end{remark}

\begin{remark}[Spins]
To better fit the geometrical framework allowed by the density matrices,
the Kohn--Sham model is presented for systems of \enquote{spinless} electrons. Real systems, as well as the numerical simulations performed at the end of this paper, include the spins. In case of systems with a positive band gap, which we consider here, everything works the same except that $\Ne$ represents the number electron \emph{pairs}, the energy reads $E(\gamma) = 2\tr(h\gamma) + F(\rho_\gamma)$ with $\rho_\gamma(\bm x) = 2\gamma(\bm x,\bm x)$ and the Kohn--Sham Hamiltonian is defined as $H_{\rho_\gamma} = 2h + 2V_{\rho_\gamma}$.
\end{remark}

\begin{remark}[Brillouin zone discretization]\label{rmk:brillouin}
Let us mention that the eigenproblems we presented are naturally equipped with periodic boundary conditions. However this introduces artificial interactions between the sample of the material in the unit cell $\Omega$ and its periodic images. In the case of a perfect crystal with Bravais lattice $\mathbb{L}$ and unit cell $\Omega$, it is recommended to choose a periodic simulation (super)cell $\omega = L\Omega$ consisting of $L^3$ unit cells, so that $\cR = L\mathbb{L}$. Using Bloch transform \cite[Section XIII.16]{reedAnalysisOperators1978}, the problem becomes in the thermodynamic limit $L\to+\infty$
\begin{equation}\label{eq:rmk_brillouin}
    \begin{cases}
		H_{{\bm k}, \rho} \phi_{i,\bm k} = \lambda_{i,\bm k} \phi_{i,\bm k}, \quad  i=1,\ldots, \Ne,\quad \bm k\in\cB,\\
  		\langle \phi_{i,\bm k},\phi_{j,\bm k} \rangle = \delta_{ij}, \quad  i,j=1,\ldots,\Ne,\quad \bm k\in\cB,\\
		\displaystyle \rho(\bm x) = \fint_{\cB} \sum_{i=1}^\Ne|\phi_{i,\bm k}(\bm x)|^2 {\rm d}{\bm k},\\
	\end{cases}
\end{equation}
where $\cB$ is the first Brillouin zone of the crystal and $H_{{\bm k}, \rho}$
are the Bloch fibers of the Kohn--Sham Hamiltonian acting on $L^2_\per(\Omega)$,
with domain $H^2_\per(\Omega)$. Finally, we consider in practice a finite subset
of $\bm k$-points in $\cB$. The geometrical framework and the bounds derived in
this paper then easily extend to the case of several $\bm k$-points in the
Brillouin zone discretization, see Section~\ref{app:brillouin_zone} in the
Supplementary Materials for details. We also refer for instance to \cite{gontierConvergenceRatesSupercell2016} for the numerical analysis of the limit $L\to+\infty$ in the case of the reduced Hartree--Fock model.
\end{remark}

\subsection{Discretization}
In order to approximate the solution to any problem in the framework of~\eqref{eq:minE} by solving the eigenproblem \eqref{eq:eig_pb}, it must be first  discretized in a finite-dimensional space. To this end, let $N$ be a positive integer, \(\cV_N\) be a finite-dimensional subspace of \(H^1_\per(\Omega)\) with ${\rm dim}(\cV_N)$ depending on $N$ such that the larger $N$, the better the Galerkin approximation of~\eqref{eq:minE}:
\begin{equation}
	\min  \left\{ E(\gamma), \; \gamma \in \cM,\; {\rm Ran}(\gamma)\subset \cV_N \right\}.
\end{equation}
The corresponding first-order optimality conditions read:
find eigenvectors \(\Phi_N\coloneqq(\phi_{i,N})_{1\leq i\leq\Ne}~\in~\cV_N^{\Ne}\) and eigenvalues \(\Lambda_N\coloneqq(\lambda_{i,N})_{1\leq i\leq \Ne}\in \R^{\Ne}\) such that
\begin{equation} \label{eq:eig-pb-disc}
	\begin{cases}
		\left( \Pi_N H_{\rho_{\gamma_{N}}} \Pi_N\right) \phi_{i,N} = \lambda_{i,N} \phi_{i,N},\quad i=1,\ldots N, \\
		\langle \phi_{i,N}, \phi_{j,N} \rangle = \delta_{ij}, \quad i=1,\ldots, \Ne,\\
        \displaystyle \gamma_{N} = \sum_{i=1}^\Ne|\phi_{i,N} \rangle \langle \phi_{i,N}|, \\
	\end{cases}
\end{equation}
where \(\Pi_N\) denotes the orthogonal projector onto \(\cV_N\) for the inner product $\langle\cdot,\cdot\rangle$.

More precisely, given \(\Ecut\in\N\), we define, for $N = \sqrt{2 \Ecut}$,
\begin{equation}\label{eq:disc_space}
    \cV_N=
    {\rm Span}\left( e_{\bm G}: \; |\bfG|\leq N\right) =
    {\rm Span}\left( e_{\bm G}: \; \frac{1}{2}|\bfG|^2\leq\Ecut\right).
\end{equation}
The parameter $\Ecut$ appearing in the previous definition is known in the materials science community as the \emph{energy cutoff}.

\subsection{Self-consistent field iterations}
Once a discretization space has been properly defined, it is possible to solve~\eqref{eq:minE} via fixed-point-like algorithms, commonly called self-consistent field (SCF) algorithms, which we briefly recall here (see \cite{cancesNumericalMethodsKohn2023} and references therein for a more detailed presentation of such algorithms). We start from an initial guess \(\Phi_{N,0} \in \cV_N^\Ne \) and, at each iteration \(m \in\N,\) we solve the following linear eigenvalue problem
\begin{equation} \label{eq:eig-pb-SCF}
	\begin{cases}
		\left( \Pi_N H_{\rho_{\gamma_{N,m}}} \Pi_N\right) \phi_{i,N,m+1} = \lambda_{i,N,m+1} \phi_{i,N,m+1}, \quad i=1,\ldots,\Ne, \\
		\langle \phi_{i,N,m+1},\phi_{j,N,m+1} \rangle = \delta_{ij}, \quad i,j=1,\ldots,N,\\
  		\displaystyle \gamma_{N,m+1} = \sum_{i=1}^\Ne|\phi_{i,N,m+1} \rangle \langle \phi_{i,N,m+1}|,\\
	\end{cases}
\end{equation}
where the Hamiltonian \(H_{\rho_{\gamma_{N,m}}}\) at the current iteration is diagonalized in order to build the density matrix $\gamma_{N,m+1}$ using the lowest $\Ne$ eigenvalues $\lambda_{1,N,m+1},\ldots,\lambda_{\Ne,N,m+1}$ and so on and so forth. By continuity arguments, the existence of a spectral gap for the infinite-dimensional problem \eqref{eq:eig_pb} ensures that, for \(N\) and \(m\) large enough, the variational eigenvalue problem at iteration \(m\) and its infinite dimensional counterpart also admit such a gap. They thus fit the assumptions required in \cite{cancesGuaranteedPosterioriBounds2020a} and error bounds can be derived for the linear eigenvalue problem \eqref{eq:eig-pb-SCF}.

\section{{\it A posteriori} analysis} \label{sec:3}
In this section we provide an \textit{a posteriori} error analysis, from which we obtain guaranteed error bounds for the ground-state energy of the continuous problem~\eqref{eq:minE}.

\subsection{Abstract analysis}

We start with a short abstract analysis, which indicates how one can obtain
guaranteed upper bounds on $E(\gamma_2)-E(\gamma_1)$ for any density matrices
$\gamma_1,\gamma_2\in\cM$, as long as the nonlinear term
$F$ is a convex functional of the density and lower bounds of the linearized
eigenvalue problems are available. We show in this section that such a bound can
be computed independently of $\gamma_1$ and under
the form $\tr\big((h+V_{\rho_{\gamma_2}}-\mu)\gamma_2\big)$,
where the constant $\mu$ is wisely chosen. Note that
we will take later on $\gamma_2 = \gamma_{N,m}$ and $\gamma_1=\gamma_\star$,
so that the aforementioned bound will ultimately end up as a bound on $E(\gamma_{N,m})
- E(\gamma_\star)$ that does not depend on $\gamma_\star$.

\begin{lemma}
	For any \(\mu \in \R,\) for any \(\gamma_1,\gamma_2\in \cM,\) there holds
	\begin{equation} \label{eq:lem31_eq}
		\begin{split}
			E(\gamma_2) - E(\gamma_1) &=
			     \tr \big((h+V_{\rho_{\gamma_2}}-\mu)\gamma_2\big)
            - \tr \big((h+V_{\rho_{\gamma_2}}-\mu)\gamma_1\big) \\ & \phantom{=}
			-\left( F(\rho_{\gamma_1}) - F(\rho_{\gamma_2} )
			-\langle F'(\rho_{\gamma_2}), \rho_{\gamma_1} - \rho_{\gamma_2} \rangle\right)
			\;.
		\end{split}
	\end{equation}
\end{lemma}
\begin{proof}
	Define the following intermediate quantity for any \(\gamma_1,\gamma_2\in\cM\):
	\begin{equation}
		J(\gamma_1,\gamma_2) \coloneqq  \tr \big((h+V_{\rho_{\gamma_2}} - \mu)\gamma_2\big)
        - \tr \big( (h+V_{\rho_{\gamma_2}} - \mu)\gamma_1\big) \;.
	\end{equation}
	Then, for any \(\mu\in \R\) and any \(\gamma_1,\gamma_2\in\cM,\) there holds, using that \(\tr(\gamma_2) = \tr(\gamma_1) = \Ne\), \eqref{eq: derivate F'(p) riesz form} and~\eqref{eq:energy_full}
	\begin{align*}
		J(\gamma_1,\gamma_2) &=
		\tr\big( (h+V_{\rho_{\gamma_2}}) \gamma_2\big) - \tr\big((h+V_{\rho_{\gamma_2}}) \gamma_1\big) \\ &=
		E(\gamma_2) + \tr(V_{\rho_{\gamma_2}} \gamma_2) - F(\rho_{\gamma_2}) - E(\gamma_1) - \tr(V_{\rho_{\gamma_2}} \gamma_1) + F(\rho_{\gamma_1}) \\ &=
		E(\gamma_2) - E(\gamma_1) + F(\rho_{\gamma_1}) - F(\rho_{\gamma_2}) - \langle F'(\rho_{\gamma_2}), \rho_{\gamma_1} - \rho_{\gamma_2} \rangle \;,
	\end{align*}
	hence the result.
\end{proof}

In order to obtain guaranteed error bounds, we now make the following, crucial, assumption.
\begin{assumption}[Convexity of the functional \(F\)]
\label{ass:convexity}
	The functional \(F\) is convex on~\(\cM\), so that
	\begin{equation}
		\forall \gamma_1, \gamma_2 \in \cM, \quad
		F(\rho_{\gamma_1}) - F(\rho_{\gamma_2}) - \langle F'(\rho_{\gamma_2}), \rho_{\gamma_1} - \rho_{\gamma_2} \rangle \ge 0.
	\end{equation}
\end{assumption}

\begin{corollary}
	\label{cor:bound}
	Let Assumption~\ref{ass:convexity} be satisfied and $\gamma_1, \gamma_2 \in \cM$ be fixed. Assume that \(\mu\in\R\) is such that
 \begin{equation}
 \label{eq:ass_positive}
 \tr \big((h+V_{\rho_{\gamma_2}}-\mu)\gamma_1\big) \ge 0.
 \end{equation}
	Then
	\begin{equation}
		E(\gamma_2) - E(\gamma_1) \le \tr\big((h+V_{\rho_{\gamma_2}} - \mu)\gamma_2\big) \;.
	\end{equation}
\end{corollary}
\begin{proof}
	Using Assumption~\ref{ass:convexity} and \(\tr\big((h+V_{\rho_{\gamma_2}}-\mu)\gamma_1\big) \ge 0\) in~\eqref{eq:lem31_eq}, we immediately obtain the result.
\end{proof}

\begin{remark}
    Note that condition \eqref{eq:ass_positive} is equivalent to $\tr((h+V_{\rho_{\gamma_2}})\gamma_1)\geq \mu\Ne$.
\end{remark}

Therefore, provided that we are able to choose \(\mu\) such that \( \tr\big((h+V_{\rho_{\gamma_2}}-\mu)\gamma_1\big) \ge 0 \; ,\) we obtain a guaranteed bound on the energy error which reads \begin{equation} E(\gamma_2) - E(\gamma_1) \le  \tr\big((h+V_{\rho_{\gamma_2}}-\mu)\gamma_2\big)\end{equation} and does not involve \(\gamma_1\). We explain in the next section how to choose this parameter $\mu$.

\subsection{Error bounds for linear eigenvalue problems}\label{sec:guaranteed}

As a preliminary step, we focus on a linear eigenvalue problem in an abstract
setting. We consider a generic, self-adjoint positive operator on
$H^1_\per(\Omega)$ with compact resolvent,
denoted by $A$, that later on will be taken as $A=H_{\rho_{\gamma_{N,m}}}$. Let \((\varphi_i)_{1\leq i\leq \Ne} \in (H^1_\per(\Omega))^{\Ne}\),
\((\varepsilon_i)_{1\leq i\leq\Ne}\in \R^\Ne\),
$\epsilon_1\leq\epsilon_2\leq\ldots\leq \epsilon_{\Ne}$, be the solutions to the
linear (infinite-dimensional) eigenvalue problem
\begin{equation} \label{eq:eig_pb_lin}
	\begin{cases}
		A \varphi_i = \varepsilon_i \varphi_i, \quad  i=1,\dots,\Ne, \\
		\langle \varphi_i, \varphi_j \rangle = \delta_{ij}, \quad  i,j=1,\dots,\Ne,\\
  		\displaystyle \gamma^0 = \sum_{i=1}^\Ne|\varphi_{i} \rangle \langle \varphi_{i}|. \\
	\end{cases}
\end{equation}
It is well known \cite[Section 2.2]{cancesPostprocessingPlanewaveApproximation2021} that $\gamma^0$ minimizes \(\gamma\mapsto\tr(A\gamma)\) over
\(\cM\).
Let $\cV_N \subset H^1_\per(\Omega)$ be a finite dimensional subspace and let \((\varphi_{i,N})_{1\leq i\leq \Ne} \in (\cV_N)^{\Ne}\), \((\varepsilon_{i,N})_{1\leq i\leq\Ne}\in \R^\Ne\), be the solutions to the linear (finite-dimensional) eigenvalue problem
\begin{equation} \label{eq:eig_pb_lin_disc}
	\begin{cases}
		\big(\Pi_NA\Pi_N\big) \varphi_{i,N} = \varepsilon_{i,N} \varphi_{i,N}, \quad  i=1,\dots,\Ne, \\
		\langle \varphi_{i,N}, \varphi_{i,N} \rangle = \delta_{ij}, \quad  i,j=1,\dots,\Ne,\\
  		\displaystyle \gamma^0_N = \sum_{i=1}^\Ne|\varphi_{{i,N}} \rangle \langle \varphi_{{i,N}}|. \\
	\end{cases}
\end{equation}
In~\cite[Theorem 5.9]{cancesGuaranteedPosterioriBounds2020a}, the authors introduce a fully computable error bound \(\eta\), which depends only on the residual, the discretization parameters and a lower bound of \(\varepsilon_{\Ne+1}\), such that
\begin{equation}\label{eq: eta_comp_error_bd}
	0 \le \sum_{i=1}^\Ne\left(\varepsilon_{i,N} - \varepsilon_{i}\right) \le \eta^2.
\end{equation}
If such a bound is available, defining
\begin{equation}\label{eq: mu computation}
	\mu_N^{\rm lb} \coloneqq \frac{1}{\Ne}\left(\sum_{i=1}^\Ne\varepsilon_{i,N} - \eta^2\right)
	\le \frac{1}{\Ne} \sum_{i=1}^\Ne \varepsilon_{i},
\end{equation}
we obtain a computable constant $\mu_N^{\rm lb}$ such that for any $\gamma\in\mathcal M$, $\tr\big((A - \mu_N^{\rm lb})\gamma\big) \ge 0$ holds, meaning~\eqref{eq:ass_positive} holds.

In order to obtain a computable $\eta^2$ satisfying \eqref{eq: eta_comp_error_bd} above, we follow~\cite{cancesGuaranteedPosterioriBounds2020a} adapted to the current settings.  This bound depends on
\begin{enumerate}[label=(\roman*)]
    \item the single eigenpair residuals operators  \(\Res(\varphi_{i,N},\varepsilon_{i,N})\)
    \cite[Definition 3.4]{cancesGuaranteedPosterioriBounds2020a};
    for each $i$ this operator is defined by its action over $\varphi\in H^1_\per(\Omega)$:
    \begin{align}
        \langle\Res(\varphi_{i,N},\varepsilon_{i,N}),\varphi\rangle_{H^{-1}_{\per}(\Omega),H^1_{\per}(\Omega)}\coloneqq \; & \varepsilon_{i,N}\langle \varphi_{i,N},\varphi\rangle \\
        & -\langle A\varphi_{i,N},\varphi\rangle_{H^{-1}_{\per}(\Omega),H^1_{\per}(\Omega)}, \nonumber
    \end{align}
    so that it can be identified with the standard residual $\varepsilon_{i,N}\varphi_{i,N} - A \varphi_{i,N}$, as a vector in $H^{-1}_\per(\Omega)$. With reasonable assumptions on the potential $V$ (see Remark~\ref{rmk:regularity}), we can assume that this residual actually belongs to $L^2_\per(\Omega)$;
    \item the \(2\)-Schatten norm, denoted by \(\norm{\cdot}_{\cG_2(\cH)}\), of the cluster residual \cite[Definition 3.5]{cancesGuaranteedPosterioriBounds2020a}
     \( \Res(\gamma_{N}) \) defined by
     \begin{align}
    &\norm{\Res(\gamma_{N})}_{\cG_2(\cH)}^2=
    \sum_{i=1}^\Ne\langle\Res(\varphi_{i,N},\varepsilon_{i,N}), A^{-1} \Res(\varphi_{i,N},\varepsilon_{i,N})\rangle_{H^{-1}_{\per}(\Omega),H^1_{\per}(\Omega)},\\
    &\norm{A^{-1/2}\Res(\gamma_{N})}_{\cG_2(\cH)}^2=\sum_{i=1}^\Ne\langle A^{-1} \Res(\varphi_{i,N},\varepsilon_{i,N}), A^{-1}\Res(\varphi_{i,N},\varepsilon_{i,N})\rangle;
\end{align}
    \item \(\underline{\varepsilon}_{\Ne+1}\), a (computable) lower bound of the exact eigenvalue \(\varepsilon_{\Ne+1}\).
\end{enumerate}
Then, by employing equations (4.1), (4.4) and (4.16) from \cite{cancesGuaranteedPosterioriBounds2020a} we obtain
\begin{equation}\label{eq:eta_energy}
    	\etaene{}^2 \coloneqq\norm{\Res(\gamma_{N})}_{\cG_2(\cH)}^2+4\varepsilon_{\Ne,N}c_{N}^2\norm{A^{-1/2}\Res(\gamma_{N})}_{\cG_2(\cH)}^2,
\end{equation}
where
\begin{equation}\label{eq:cN}
    c_{N}=\left(1-\frac{\varepsilon_{\Ne,N}}{\underline{\varepsilon}_{\Ne+1}}\right)^{-1}.
\end{equation}
There are two difficulties in the evaluation of this bound: it involves the inversion of the operator $A$ and the computation of a lower bound on $\varepsilon_{\Ne+1}$.
The inversion can be circumvented by recasting it as a linear systems of equations, which is still computationally expensive. This is the focus of the next section, after a short remark on the use of negative Sobolev norms. The computation of the lower bound will be commented in Section~\ref{sec:4}.

\begin{remark}[Negative Sobolev norms]
   The computation of the error bounds using negative Sobolev norms of the residuals
   in~\cite{cancesGuaranteedPosterioriBounds2020a} relies on
    additional assumptions on the linear operator $A$, for instance $A \geq -\frac12\Delta + 1$, a much stronger condition than $A\geq 0$.
 This is fine from a mathematical point of view since shifting by a constant $\sigma > 0$ can ensure that this condition holds. However, in our numerical experiments it appeared that such shifts were so significant that, even for systems with large gaps, the \emph{relative} gap  $\frac{\varepsilon_{\Ne,N}+\sigma}{\underline{\varepsilon}_{\Ne+1}+\sigma}$ then becomes very close to 1, making in turn the constant $c_N$ from \eqref{eq:cN} of order $10^6$, thus rendering the bound unusable. Second, \cite[Theorem 5.9]{cancesGuaranteedPosterioriBounds2020a} actually relies not only on $A \geq -\frac12\Delta + 1$, but also on $A^2 \geq (-\frac12\Delta + 1)^2$, which is \emph{not} implied by the first condition if $A$ and $\Delta$ do not commute.
\end{remark}

\begin{remark}[Regularity of the residuals]\label{rmk:regularity}
    The regularity of the potential $V$ has a direct impact on the regularity of the orbitals $\varphi_i$ and density $\rho$. For instance, it was shown in \cite{cancesNumericalAnalysisPlanewave2012a} for the LDA exchange-correlation functional, that if $V \in H^s_\per(\Omega)$ for $s>3/2$, then the orbitals and the density are in $H^{s+2}_\per(\Omega)$, a valid framework for Troullier--Martins pseudopotentials \cite{troullierEfficientPseudopotentialsPlanewave1991}. For GTH pseudopotentials \cite{goedecker1996separable,hartwigsen1998relativistic}, used in the numerical simulations performed with DFTK in Section~\ref{sec:4}, the pseudopotentials are actually real-analytic, a property that translates to the orbitals, see \cite{cancesPrioriErrorAnalysis2024}. For all these reasons, it is reasonable to assume that the residuals have actually at least $L^2_\per(\Omega)$ regularity.
\end{remark}

\subsection{Practical strategies for efficient computation of error bounds based on operator splitting}\label{sec:practical}

We now discuss some alternative strategies which do not require to solve a full linear system for computing \eqref{eq:eta_energy}.
We assume $A = -\frac12\Delta + V$ for some linear potential $V \in L^2_\per(\Omega;\R)$ such that $A>0$, as the Hamiltonian at a given SCF iteration will be of this form.
Recall that $\Pi_N$ is the orthogonal projection on the final dimensional subspace $\cV_N\subset H^1_\per(\Omega)$, with $\Pi_N^\perp$ being the projection on the corresponding orthogonal complement.
In the decomposition $\cH = \cV_N \oplus \cV_N^\perp$, the operator $A$ can be written in block representation as
\begin{equation}
    A=\begin{bmatrix}
        \Pi_NA\Pi_N &\Pi_N A\Pi_N^\perp\\\Pi_N^\perp A\Pi_N & \Pi_N^\perp A \Pi_N^\perp
    \end{bmatrix}.
\end{equation}
Defining $\langle V\rangle \in \R$ as the average value of the potential $V$ over the unit cell and setting
\begin{equation}
    H_0=\begin{bmatrix}
        \Pi_N \left(-\frac{1}{2}\Delta+V\right)\Pi_N & 0 \\ 0 & \Pi_N^\perp \left(-\frac{1}{2}\Delta+\langle V\rangle\right)\Pi_N^\perp
    \end{bmatrix},
\end{equation}
together with
\begin{equation}
    W=\begin{bmatrix}
        0 &\Pi_N V\Pi_N^\perp\\\Pi_N^\perp V\Pi_N & \Pi_N^\perp \big(V-\langle V\rangle\big) \Pi_N^\perp
    \end{bmatrix},
\end{equation}
we have $A=H_0+W$. Here, the Laplace operator does not appear in the
off-diagonal blocks because it commutes with the projections $\Pi_N$ and
$\Pi_N^\perp$.
Under the assumption $\norm{H_0^{-1}W}<1$, $A^{-1}$ admits the Neumann expansion
\begin{equation} \label{ref:neumann}
    A^{-1}=\sum_{n=0}^{+\infty}(-H_0^{-1}W)^nH_0^{-1}.
\end{equation}
Therefore $A^{-1}$ can be approximated, among others, by the zeroth-order term in the series:
\begin{equation}\label{eq: zeroth_ord_bnd}
    A^{-1}=H_0^{-1}+\mathcal{O}(\norm{H_0^{-1}W}),
\end{equation}
or the first-order term of the series:
\begin{equation}\label{eq: first_ord_bnd}
    A^{-1}=H_0^{-1}-H_0^{-1}WH_0^{-1}+\mathcal{O}(\norm{H_0^{-1}W}^2).
\end{equation}

This leads to two corresponding approximations for the estimator $\eta$ defined by \eqref{eq:eta_energy}. Using the notation $r_{i,N}=\Res(\varphi_{i,N},\varepsilon_{i,N})$ for the sake of clarity, this bound reads:
\begin{equation}\label{eq:eta_energy1}
    \etaene{0}^2  \coloneqq \sum_{i=1}^\Ne\langle r_{i,N}, H_0^{-1} r_{i,N}\rangle
    +4\varepsilon_{\Ne,N}c_{N}^2
    \sum_{i=1}^\Ne\langle H_0^{-1} r_{i,N}, H_0^{-1}r_{i,N}\rangle,
\end{equation}
and
\begin{equation}\label{eq:eta_energy2}
    \begin{split}
        \etaene{1}^2 \coloneqq & \sum_{i=1}^\Ne\langle r_{i,N}, \left( H_0^{-1}-H_0^{-1}WH_0^{-1} \right)  r_{i,N}\rangle \\ & +4\varepsilon_{\Ne,N}c_{N}^2
         \sum_{i=1}^\Ne\langle \left( H_0^{-1}-H_0^{-1}WH_0^{-1} \right)  r_{i,N}, \left( H_0^{-1}-H_0^{-1}WH_0^{-1} \right) r_{i,N}\rangle.
     \end{split}
\end{equation}

Using either one of these two bounds on the eigenvalue differences, one can compute a bound on the energy, following~\eqref{eq: mu computation}, replacing $\eta$ respectively by $\etaene{0}$ or $\etaene{1}$.
By introducing this approximation, we have replaced the full operator inversion by the inversion of the block diagonal operator $H_0$ and a multiplication by the operator $W$. Note that the block $\Pi_N^\perp \left(-\frac{1}{2}\Delta+\langle V\rangle\right)\Pi_N^\perp$ being diagonal in Fourier representation greatly simplifies its inversion.
Also, the block $\Pi_N \left(-\frac{1}{2}\Delta+V\right)\Pi_N$ is dense, but its inversion is performed on the finite dimensional space $\cV_N$. In fact, since \(\Pi_{N}r_{i,N}=0 \) (if we assume the finite-dimensional linear eigenvalue problem \eqref{eq:eig_pb_lin_disc} is solved exactly), the computation of the zeroth-order approximation is obtained by the inversion of a diagonal system in $\cV_N^\perp$ and thus is very efficient. The price of this approach is the introduction of an additional, \emph{a priori} uncontrollable, source of error, originating from the truncation of the Neumann series.

We now obtain an upper bound for the remainder in~\eqref{eq: zeroth_ord_bnd} or~\eqref{eq: first_ord_bnd}, controlling the additional error introduced above. This will result in two more additional bounds that are mathematically guaranteed. We begin by writing the preconditioned residual $A^{-1}r_{i,N}$ as
\begin{equation} \label{eq:3-4}
	A^{-1}r_{i,N} = \sum_{n=0}^{+\infty}(-H_{0}^{-1}W)^nH_{0}^{-1}r_{i,N},
\end{equation}
and split it into two parts, depending on the number of terms in the approximation, which we denote by \(L\). Notice $L=0$ corresponds to ~\eqref{eq: zeroth_ord_bnd} and $L=1$ corresponds to~\eqref{eq: first_ord_bnd}. We have \(A^{-1}r_{i,N} = \chi_{i,N}+ e_{i,N}\) with
\begin{equation}
	\chi_{i,N} = \sum_{n=0}^{L}(-H_{0}^{-1}W)^nH_{0}^{-1}r_{i,N}
\end{equation}
and
\begin{equation}
    e_{i,N} = \sum_{n=L+1}^{+\infty}(-H_{0}^{-1}W)^nH_{0}^{-1}r_{i,N}.
\end{equation}
The norm of the remainders can then be estimated as
\begin{equation} \label{eq:3-1}
    \begin{split}
    	\|e_{i,N}\| &= \left\|\sum_{n=L+1}^{+\infty}(-H_{0}^{-1}W)^nH_{0}^{-1}r_{i,N}\right\|
        \\ &=
    	\left\|\sum_{n=0}^{+\infty} (-H_{0}^{-1}W)^n (-H_{0}^{-1}W)^{L+1} H_{0}^{-1}r_{i,N}\right\| \\ & \leq
    	\frac{\|(-H_{0}^{-1}W)^{L+1} H_{0}^{-1}r_{i,N}\|}{1-\|H_{0}^{-1}W\|} \leq
    	\frac{\|(-H_{0}^{-1}W)^{L+1}\| \|H_{0}^{-1}r_{i,N}\|}{1-\|H_{0}^{-1}W\|}
        \eqqcolon \tilde{e}_{i,N,L}.
    \end{split}
\end{equation}

For any $i=1,\ldots, \Ne$, we can now bound the preconditioned residuals appearing in \eqref{eq:eta_energy} as follows:
\begin{equation} \label{eq:3-2}
    \begin{split}
    	\langle r_{i,N} , A^{-1}r_{i,N}\rangle &=
    	\langle r_{i,N} , \chi_{i,N}+ e_{i,N}\rangle \\ &\leq
    	\langle r_{i,N} , \chi_{i,N}\rangle + \|r_{i,N}\| \|e_{i,N}\|,
     \end{split}
\end{equation}
and
\begin{equation} \label{eq:3-3}
    \begin{split}
    	\langle A^{-1}r_{i,N} , A^{-1}r_{i,N}\rangle &=
    	\langle \chi_{i,N}+ e_{i,N} , \chi_{i,N}+ e_{i,N}\rangle \\ &\leq
    	\langle \chi_{i,N}, \chi_{i,N}\rangle + 2\|e_{i,N}\| \|\chi_{i,N}\| + \|e_{i,N}\|^2.
    \end{split}
\end{equation}

Next, we sum \eqref{eq:3-2} and \eqref{eq:3-3} over $i=1,\dots,\Ne$, substitute in \eqref{eq:eta_energy}, \eqref{eq:3-1} and identify the terms corresponding to~\eqref{eq:eta_energy1} and~\eqref{eq:eta_energy2} to obtain fully guaranteed upper bounds on $\eta^2$. These bounds read, for the zeroth-order approximation,
\begin{equation}\label{eq:eta_energy1_g}
    \eta^2 \leq \etaene{0}^2 + \sum_{i=1}^\Ne \|r_{i,N}\| \tilde{e}_{i,N,0} + 4\varepsilon_{\Ne,N}c_{N}^2\big(2\tilde{e}_{i,N,0} \|H_0^{-1}r_{i,N}\| + \tilde{e}_{i,N,0}^2\big) \eqqcolon \etaene{0,\rm g}^2,
\end{equation}
and, for the first-order approximation,
\begin{equation}\label{eq:eta_energy2_g}
    \begin{split}
        \eta^2 &\leq \etaene{1}^2 + \sum_{i=1}^\Ne \|r_{i,N}\| \tilde{e}_{i,N,1} \\ &+ 4\varepsilon_{\Ne,N}c_{N}^2\big(2\tilde{e}_{i,N,1} \|( H_0^{-1}-H_0^{-1}WH_0^{-1})r_{i,N}\| + \tilde{e}_{i,N,1}^2\big) \eqqcolon \etaene{1,\rm g}^2.
    \end{split}
\end{equation}

Regarding a way to estimate the operator norm \(\|H_{0}^{-1}W\|\), we refer to
Section~\ref{app:opnorm} in the Supplementary Materials.
We now show how to compile everything into \emph{a posteriori} error bounds for the nonlinear problem.

\subsection{Error bounds for the nonlinear problem}

We now apply the error bounds obtained for a linear operator $A$ to the nonlinear problem of interest in this article, namely~\eqref{eq:eig_pb}.
In practice, for a fixed $N$, $\gamma_N$ is not directly computable but is rather obtained as the limit of a sequence $(\gamma_{N,m})_{m\in\N}$ typically generated by SCF iterations \eqref{eq:eig-pb-SCF}.
To this end, we follow Section~\ref{sec:guaranteed} by applying Corollary~\ref{cor:bound} with a global minimizer $\gamma_\star$ of \eqref{eq:minE} (resp. $\gamma_{N,m}$ from \eqref{eq:eig-pb-SCF}) in place of $\gamma_1$ (resp. $\gamma_2$), $A = H_{\rho_{\gamma_{N,m}}}$ (with $A$ shifted by a positive constant if necessary to guarantee $A>0$) and $\varepsilon_{i,N} = \lambda_{i,N,m+1}$. In other words, we compute \( \mu^{\rm lb}_{N,m+1}\) as a lower bound of the mean of the eigenvalues from the (infinite-dimensional) eigenvalue problem \eqref{eq:eig_pb_lin} with $A = H_{\rho_{\gamma_{N,m}}}$, using~\eqref{eq: mu computation} with one of the proposed bounds on $\eta$, \emph{e.g.} $\etaene{0}$ or $\etaene{1}$.
This bound then reads
\begin{equation}
	0\leq E(\gamma_{N,m}) - E(\gamma_\star) \le \tr\big((H_{\rho_{\gamma_{N,m}}} -\mu^{\rm lb}_{N,m+1})\gamma_{N,m}\big),
\end{equation}
and yields the following certification of the energy at iteration \(m\) of the SCF:
\begin{equation}
	E(\gamma_\star) \in  \Big[E(\gamma_{N,m}) - \tr\big((H_{\rho_{\gamma_{N,m}}} -\mu^{\rm lb}_{N,m+1})\gamma_{N,m}\big), E(\gamma_{N,m})\Big].
\end{equation}

We can then also separate the error bound \(\tr\big((H_{\rho_{\gamma_{N,m}}}-\mu^{\rm lb}_{N,m+1})\gamma_{N,m}\big)\) into two parts: one depending mainly on the discretization and the other depending mainly on the number of performed SCF iterations. More precisely, define
\begin{equation}\label{eq: discretization bound}
	\err_{N,m}^{\rm disc} \coloneqq\tr\big((H_{\rho_{\gamma_{N,m}}} -\mu^{\rm lb}_{N,m+1})\gamma_{N,m+1}\big)\geq 0
\end{equation}
and
\begin{equation}\label{eq: SCF bound}
	\begin{split}
		\err_{N,m}^{\rm SCF} &\coloneqq
        \tr(H_{\rho_{\gamma_{N,m}}}\gamma_{N,m}) - \tr(H_{\rho_{\gamma_{N,m}}}\gamma_{N,m+1}) \geq 0. \\
	\end{split}
\end{equation}

Then, as $\tr(\gamma_{N,m+1}) = \tr(\gamma_{N,m})$, we naturally have that $\tr\big((H_{\rho_{\gamma_{N,m}}} -\mu^{\rm lb}_{N,m+1})\gamma_{N,m}\big) = \err_{N,m}^{\rm disc} + \err_{N,m}^{\rm SCF} $. Putting things together, we proved the following theorem.
\begin{theorem}[Fully guaranteed error bound on the energy]
    Under Assumption~\ref{ass:convexity} and assuming that, at iteration $m$ of the SCF algorithm in $\cV_N$, \eqref{eq:ass_positive} holds with $\mu = \mu^{\rm lb}_{N,m+1}$, then we have
    \begin{equation}\label{eq: full bound}
    	\boxed{E(\gamma_{N,m}) - E(\gamma_\star) \le{\normalfont\err}_{N,m}^{\rm disc} + {\normalfont\err}_{N,m}^{\rm SCF},}
    \end{equation}
    where ${\normalfont\err}_{N,m}^{\rm disc}$ and ${\normalfont\err}_{N,m}^{\rm SCF}$ are respectively defined by \eqref{eq: discretization bound} and \eqref{eq: SCF bound}.
\end{theorem}

We remark that (i) \(\err_{N,m}^{\rm SCF}\) goes to zero as \(m\) goes to infinity, provided that the SCF algorithm does converge, and (ii) \(\err_{N,m}^{\rm disc}\) tends to zero as the discretization space is enlarged, provided that the limit of \(\cV_N\) when \(N \rightarrow +\infty\) is \(H_\per^1(\Omega)\) in the sense that $\forall \phi\in H^1_\per(\Omega)$, $\lVert{\phi-\Pi_N\phi}_{H^1_\per(\Omega)}\rVert\rightarrow 0$ as $N\rightarrow\infty$, and that \(\mu^{\rm lb}_{N,m+1}\) is well chosen, for instance as explained above. We end this paper with a series of test cases, in 1D and 3D, where these bounds and their approximations are applied and compared in term of computational cost and accuracy.

\section{Numerical results}\label{sec:4}

\subsection{Numerical framework}\label{subsec: numerical_framework}

All the simulations presented in this section are realized with the DFTK software, a recent \texttt{Julia} package to perform planewave DFT calculations \cite{herbstDFTKJulianApproach2021}.
DFTK uses the planewave basis introduced in Section~\ref{subsec:kohn_sham}, through a
discretization parameter $\Ecut \coloneqq N^2/2$. The nonlinear eigenproblems~\eqref{eq:eig-pb-disc} for the Kohn--Sham Hamiltonian are solved iteratively using the SCF algorithm described previously, with proper tuning to ensure convergence (Anderson/DIIS acceleration, density mixing, damping, etc.), with a tolerance on the $L^2$-norm between two successive densities set to $10^{-10}$  to stop the iterations. Notice that the linear eigenproblems are solved at each iteration of the SCF procedure with a LOBPCG solver (see \emph{e.g.}\ \cite{nottoliRobustOpensourceImplementation2023}). We assume in this paper that these linear eigenproblems are exactly solved within the variational approximation space $\cV_N$, the treatment of the numerical linear algebra error being left for future work.

The simulations are performed within the periodic reduced Hartree--Fock(rHF) model, for which $E_{\rm xc}(\rho)=0$ (see \emph{e.g.}\ \cite{cances_new_2008,Solovej_1991} for the mathematical properties of the rHF model). Assumption~\ref{ass:convexity} is thus satisfied, as the functional $\rho\mapsto\cD(\rho,\rho)$ is (strictly) convex \cite{gontierConvergenceRatesSupercell2016} on \emph{e.g.} the affine space $\lbrace\rho\in L^2_\per(\Omega): \int_\Omega\rho=\Ne\rbrace.$

Each simulation is performed in a reference space, defined by a parameter $\Ecutref$, and a computation space, defined by a parameter $\Ecut<\Ecutref$. The reference space is assumed to account for the full space and is used to compute the reference solutions, as well as the residuals in various norms. The computation space is used to perform the actual calculation and we seek to bound the error on the energy due to the variational approximation of the Kohn--Sham equations \eqref{eq:KS-DFT-per} in this space.

In all the simulations below, we track the error bound on the energy difference $E(\gamma_{N,m}) - E(\gamma_\star)$ by computing $\err^{\rm SCF}_{N,m}$ and $\err^{\rm disc}_{N,m}$ as in \eqref{eq: full bound}, the latter being evaluated with the different $\eta$'s highlighted in Sections~\ref{sec:guaranteed} and \ref{sec:practical}. For convenience, we recall in Table~\ref{tab:equation_reference_table} the different possibilities, together with their properties and computational cost. Note also that, for the sake of simplicity, the constant $\underline{\varepsilon_{\Ne+1}}$ appearing in $c_N$ \eqref{eq:cN} is taken as the variational approximation $\varepsilon_{\Ne+1,N}$.
Other methods to estimate this lower bound are proposed in~\cite{cancesGuaranteedPosterioriBounds2020a}.

\begin{table}[h!]
\centering
\caption{Names of the bounds and corresponding expressions for the different ways of computing $\eta$ in \eqref{eq: mu computation}. Each of these expression requires $A=H_{\rho_{\gamma_{N,m}}} > 0$ and has a different computational cost.}
\label{tab:equation_reference_table}
\footnotesize
\begin{tabular}{@{}cccccc@{}}
\toprule
\textbf{Name}  & \textbf{Nota-} & \textbf{Fully} & \textbf{Equation} & \textbf{Condition} & \textbf{Computational cost} \\
 & \textbf{tion} & \textbf{guaranteed} & &\\ \midrule
full-inversion & $\etaene{}$  & yes & \eqref{eq:eta_energy} & $A>0$ & full-inversion of $A$ in $\cH$      \\ \midrule
zeroth-order   & $\etaene{0}$ & no & \eqref{eq:eta_energy1} & $A>0$ & diagonal inversion in $\cV_N^\perp$ \\ \midrule
zeroth-order  & $\etaene{0,\rm g}$ & yes & \eqref{eq:eta_energy1_g} & $A>0$ & diagonal inversion in $\cV_N^\perp$ \\
guaranteed  & & & & $\|H_0^{-1}W\|<1$ & and remainder estimation   \\ \midrule
first-order & $\etaene{1}$ & no & \eqref{eq:eta_energy2} & $A>0$ & full-inversion in $\cV_N$ \\ \midrule
first-order &  $\etaene{1,\rm g}$ & yes & \eqref{eq:eta_energy2_g} & $A>0$ & full-inversion in $\cV_N$ \\
guaranteed  & & & & $\|H_0^{-1}W\|<1$ & and remainder estimation   \\\bottomrule
\end{tabular}
\end{table}

\subsection{1D toy model}

We first present simulations obtained with a 1D toy model ($\Omega=(0,10)$), for which reference solutions with very high accuracy can be obtained with a moderate computational cost. The reference basis is built with $\Ecutref = 1000\ \Ha$ and the calculation basis with $\Ecut = 400\ \Ha$. The potential $V$ we use is defined by its Fourier coefficients as:
\begin{equation}
     \forall\ G \in 2\pi\Z, \quad \hat V_G = \begin{cases}
            1 & \mbox{if } G=0, \\
            \frac{\omega_G}{|G|^{1.1}} & \mbox{if } 0 < |G| \leq 100, \\
            \frac{1}{|G|^{1.1}} & \mbox{else}, \\
        \end{cases}
\end{equation}
where $(\omega_G)_{0 < |G| \leq 100}$ are independent random variables, uniformly distributed between $-10$ and $10$. We then solve iteratively the Kohn--Sham equations for $\Ne=3$ on the unit cell $\Omega$ with periodic boundary conditions. Note that we obtained qualitatively similar results with other potentials with comparable Sobolev regularity.

In Figures~\ref{fig:1D_nonopt_shift} and \ref{fig:1D_optimized_shift}, we display the bounds from Table~\ref{tab:equation_reference_table}: the full operator inversion, the zeroth-order and first-order approximations, both with and without the estimation of the remainder terms in the truncation of the Neumann series. Note that, when computing the error bound with $\etaene{0,{\rm g}}$ and $\etaene{1,\rm g}$ at every step of the SCF cycle, not only has the Hamiltonian to be shifted to ensure positivity, but also to satisfy $\|H_0^{-1} W \|<1$. One then realizes that, while both the zeroth- and first-order approximations already give very satisfactory results in Figure~\ref{fig:1D_nonopt_shift} (left), adding the remainder terms in the Neumann series to make the bound fully guaranteed worsen the error estimation by about an order of magnitude (Figure~\ref{fig:1D_optimized_shift}). We also optimized the shift $s$ in order to make the estimation of the remainder terms \eqref{eq:3-1} as small as possible. The optimal shift is computed by a dichotomy to find a zero of the derivative of $\texttt{err}^{\rm disc}_{N,m}$ with respect to $s$, while keeping the constraint $\|H_0^{-1} W \|<1$. This clearly improves the resulting bound, which is guaranteed, but the final accuracy is still far from being satisfying (Figure~\ref{fig:1D_optimized_shift}). As a conclusion, one sees that the best error estimation, even though it is not guaranteed, is given by the zeroth-order bound, without estimating the remainders. Moreover, in Figure~\ref{fig:1D_nonopt_shift} (right), the transition between the two contributions ($\normalfont\err^{\rm SCF}_{N,m}$ and $\err^{\rm disc}_{N,m}$) to the error clearly appears.

In Table~\ref{tab:1D}, we compute the effective ratio between $E(\gamma_{N,m})-E(\gamma_\star)$ and the upper bound $\normalfont\err^{\rm SCF}_{N,m} + \err^{\rm disc}_{N,m}$ for all the $\eta$'s mentioned above. We expect this ratio to be close to 1 when the bound is accurate, which confirms the observations made before. Note that the guaranteed bound obtained with the full-inversion is well approximated by the first-order bound, as there is an additional term taken into account in the Neumann series. The zeroth-order bound also seems to approximate better the true error than the full-inversion bound: this is due to the truncation of the Neumann series used to design this bound, and one can not expect this observation to hold in general.

\begin{figure}[p!]
    \centering
    \includegraphics[width=0.49\linewidth]{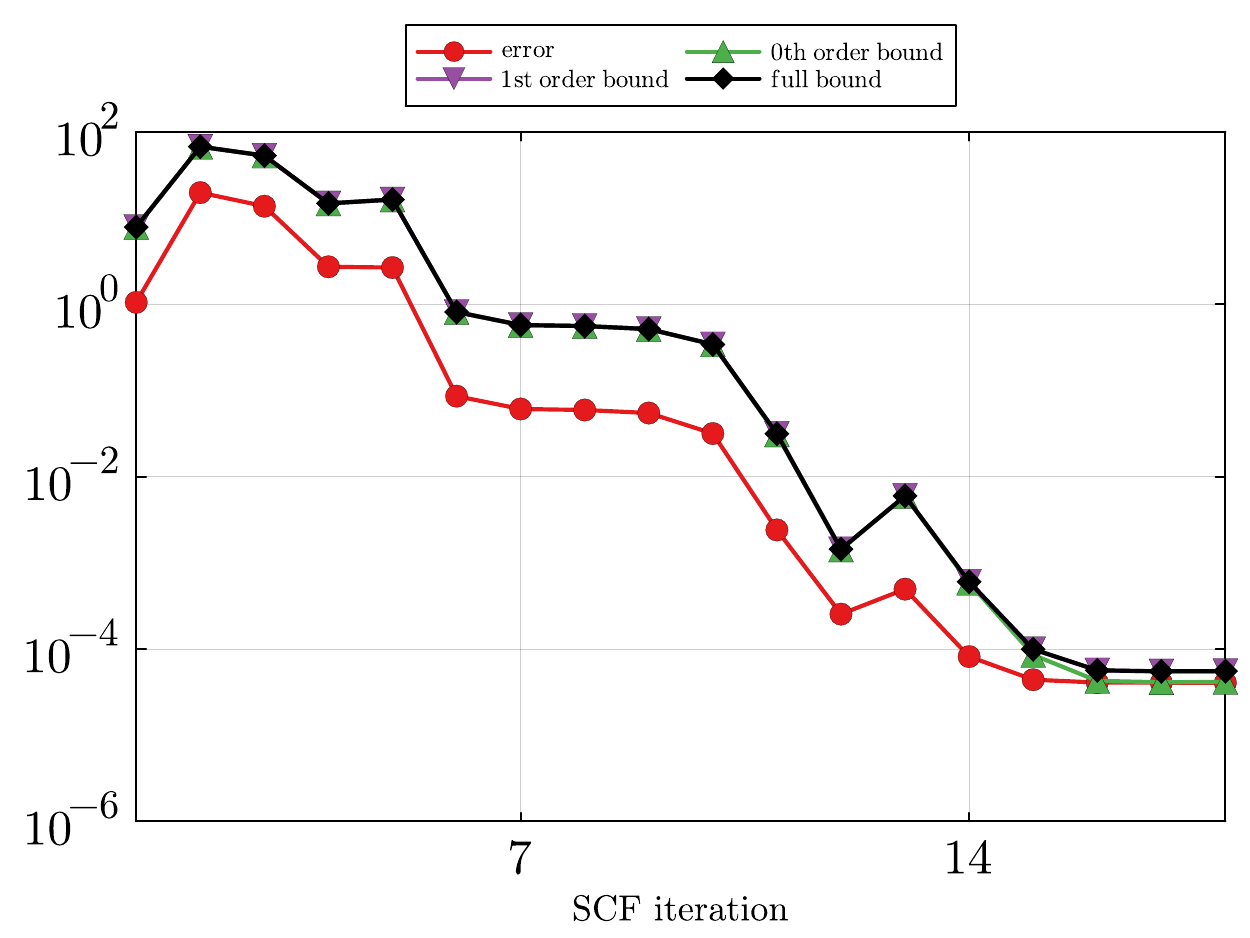}\hfill
    \includegraphics[width=0.49\linewidth]{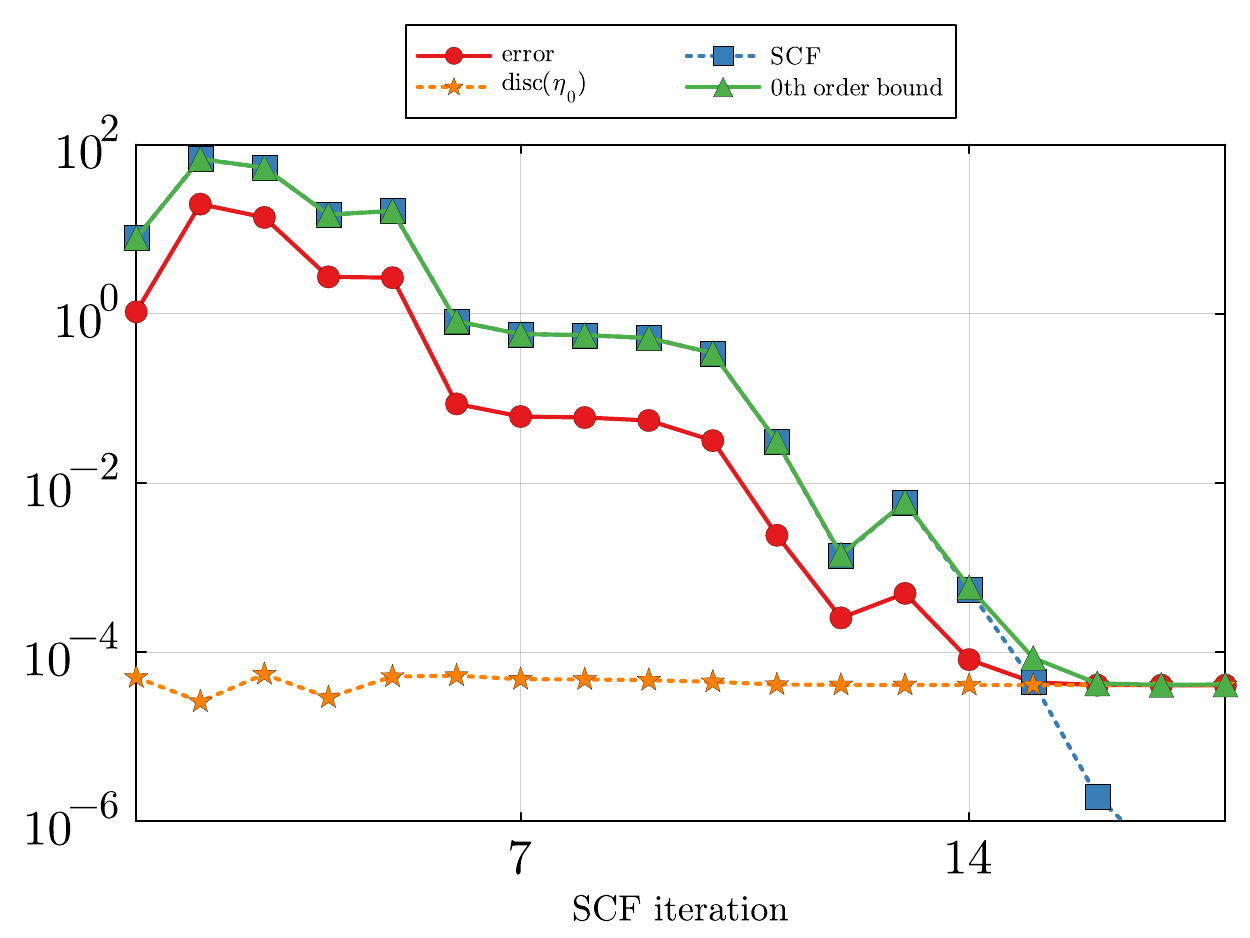}
    \caption{Tracking of the error $E(\gamma_{N,m}) - E(\gamma_\star)$ for a 1D toy model. (Left) full-inversion bound and its zeroth- and first-order approximations. (Right) Zeroth-order bound and its splitting between SCF and discretization contributions, as in \eqref{eq: full bound}. \\ }
    \label{fig:1D_nonopt_shift}
\end{figure}

\begin{figure}[p!]
    \centering
    \includegraphics[width=0.49\linewidth]{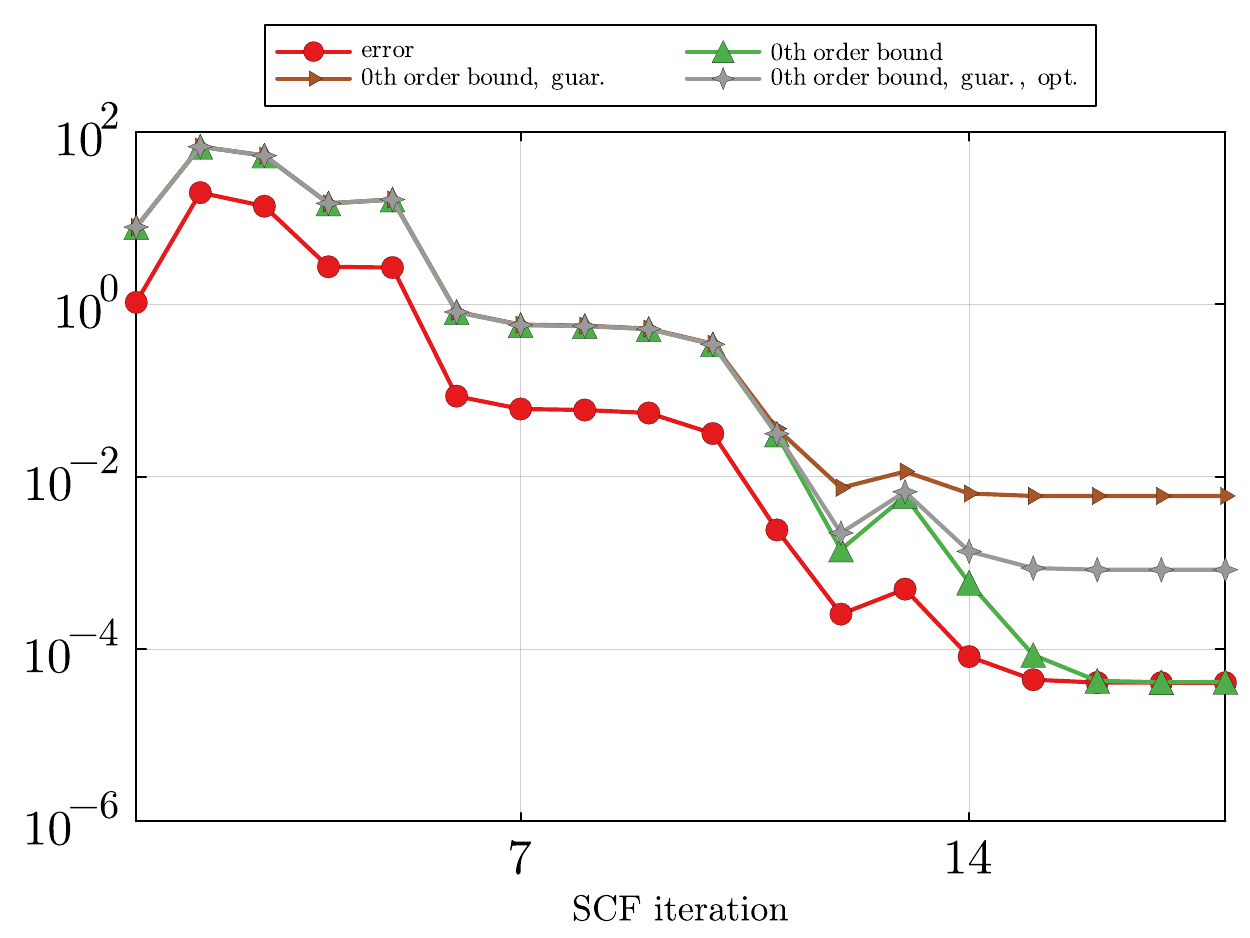}\hfill
    \includegraphics[width=0.49\linewidth]{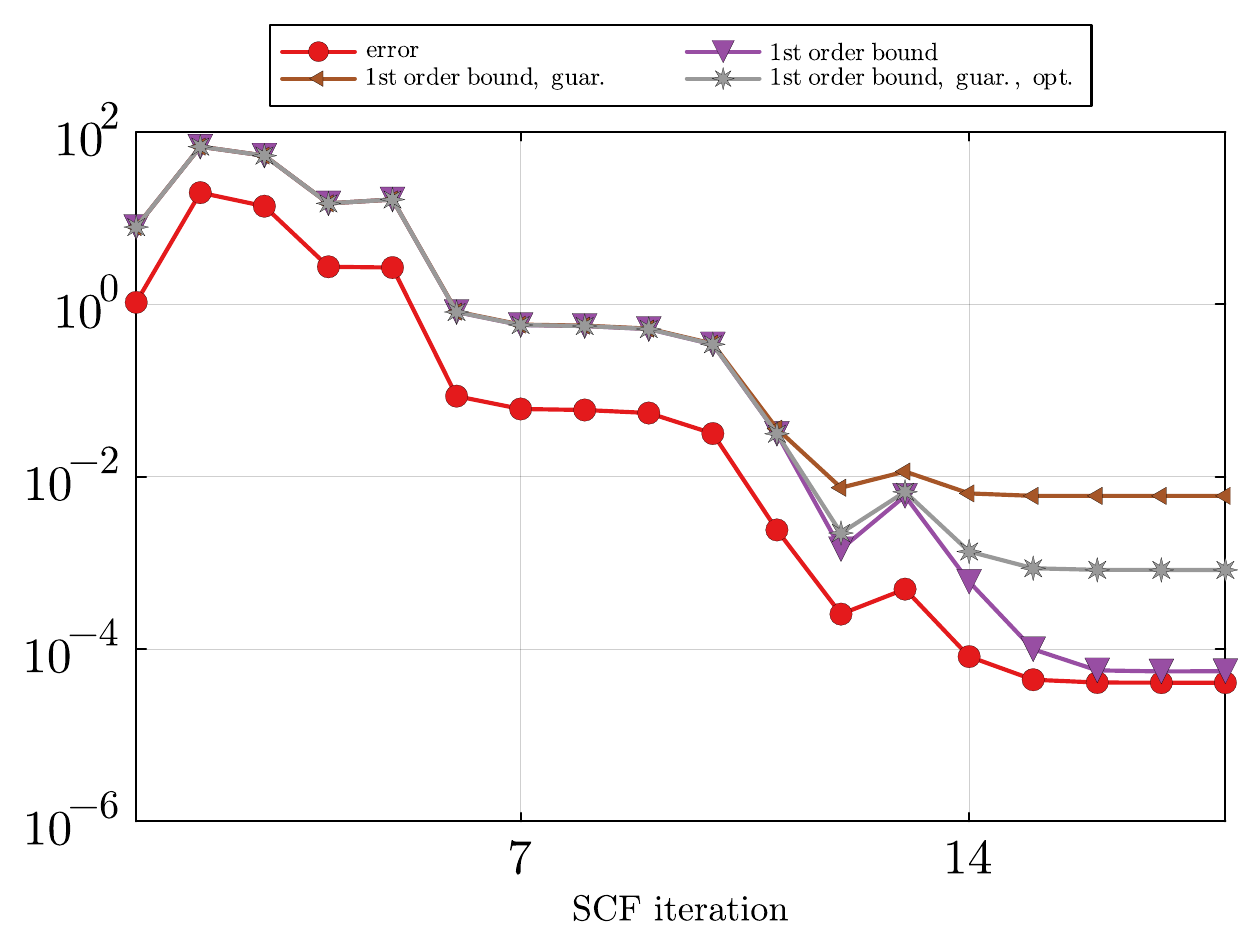}
    \caption{Tracking of the error $E(\gamma_{N,m}) - E(\gamma_\star)$ for a 1D toy model, with the zeroth-order bound (left) and first-order bound (right). We added the associated guaranteed bounds by estimating the remainders and their optimization with respect to the shift.
    }
    \label{fig:1D_optimized_shift}
\end{figure}

\begin{table}[p!]
    \centering
    \caption{Ratio between the error $E(\gamma_{N,m})-E(\gamma_\star)$ and the upper bound $\normalfont\err^{\rm SCF}_{N,m} + \err^{\rm disc}_{N,m}$ for different $\eta$'s, for a 1D toy model. The closer to 1 the ratio, the better the bound.}
    \label{tab:1D}
    \begin{tabular}{@{}cccccccc@{}}
        \toprule
        SCF it. & $\etaene{0}$ & $\etaene{1}$ & $\etaene{}$ & $\etaene{0,\rm g}$ & $\etaene{1,\rm g}$ & $\etaene{0,{\rm g},{\rm opt}}$ & $\etaene{1,{\rm g},{\rm opt}}$ \\ \midrule
        1       & 7.47740 & 7.47740 & 7.47740 & 7.48102 & 7.48106 & 7.47805 & 7.47800 \\
        5       & 6.16377 & 6.16377 & 6.16377 & 6.16581 & 6.16585 & 6.16424 & 6.16424 \\
        10      & 10.8366 & 10.8368 & 10.8368 & 1.04030 & 11.0411 & 10.8586 & 10.8585 \\
        15      & 1.94103 & 2.25435 & 2.25719 & 135.935 & 136.214 & 19.7679 & 19.6773 \\
        20      & 1.00401 & 1.34280 & 1.34587 & 147.209 & 147.513 & 20.4755 & 20.3783 \\ \bottomrule
    \end{tabular}
\end{table}

\subsection{3D insulating systems}

We now present some numerical illustrations of the error estimation developed in this paper for real 3D systems. In addition to the previous settings, we use the Goedecker--Teter--Hutter (GTH) pseudopotentials
\cite{goedecker1996separable,hartwigsen1998relativistic}, already implemented in DFTK, see Remark~\ref{rem4.1} below. We then consider two physical systems:
\begin{itemize}
    \item a Silicon crystal, first with a single $\bm k$-point only (the $\Gamma$ point), then with eight $\bm k$-points.
    \item a Hydrogen-Fluoride molecule with eight $\bm k$-points.
\end{itemize}

\begin{remark}[pseudopotentials norms] \label{rem4.1}
    The calculations from Section~\ref{app:opnorm} in the Supplementary
    Materials can easily be extended to pseudopotentials in the Kleinmann--Bylander form \cite{kleinmanEfficaciousFormModel1982}: in addition to the $L^\infty$ norm of the local contribution, one just has to add the largest of the projection coefficients from the nonlocal contribution.
\end{remark}

\subsubsection{Silicon (Si) crystal}
For the Si crystal, a semiconductor with a positive band gap, we set
$\Ecutref=400\ \Ha$, and $\Ecut=150\ \Ha$. For the $\bm k$-point sampling (see
Section~\ref{app:brillouin_zone} in the Supplementary Materials), we use a single $\bm k$-point for the first test and eight $\bm k$-points for the second one.

In Figure~\ref{fig:Si_1kpt} (left), we  display, for the first case with a single $\bm k$-point, the error with respect to the reference energy as well as the different error bounds computed from: (i) the energy norm \eqref{eq:eta_energy} using the full inverse, (ii) the zeroth-order approximation \eqref{eq: zeroth_ord_bnd} and (iii) the first-order approximation \eqref{eq: first_ord_bnd}. Even though the last two bounds are not mathematically guaranteed because of the truncation of the Neumann series, the accuracy of the estimation is very satisfying. In Figure~\ref{fig:Si_1kpt} (right), the transition from a SCF dominating error to a discretization dominating error also clearly appears. As for the 1D toy model, we added in Figure~\ref{fig:Si_1kpt_optimized_shift} the estimation of the remainders together with their optimization with respect to the shift. Again, the resulting bound, now mathematically guaranteed, is far from being effective.

In Figure~\ref{fig:Si_8kpt}, we present the results for simulations for the same system, now with eight $\bm k$-points. First, note that this case is not suitable for the rHF model as Silicon behaves as a metallic system in this approximation: one usually needs to introduce a numerical smearing and fractional occupation numbers (see for instance \cite{cancesNumericalQuadratureBrillouin2020} and references therein) to ensure the convergence of the SCF procedure. To avoid this, we use the following workaround. As Silicon is expected to be a semi-conducting system when using the LDA approximation for the exchange-correlation energy (which is not convex), we first run a simulation for this model. Then, we extract the effective (\emph{i.e.} converged) exchange-correlation potential $V_{\rm xc}[\rho]$, and we use it in the external, linear, potential $V$ before running the SCF algorithm for the rHF model. Again, our main observation here is that there is not a significant difference between using the zeroth-order approximation, the first-order approximation or the full-inversion of the operator to compute the energy norm, with a clear transition between the SCF error regime to the discretization error regime. We also observed in our simulations that estimating the remainders of the Neumann series yields bounds that are not accurate enough to be of any practical interest.

Finally, we tabulate in Table~\ref{tab:ratio_comparison} the ratio between the zeroth-order bound or the full-inversion bound and the targeted error: note the difference between the accuracy of the bounds when using one or eight $\bm k$-points. This is due to the chosen linear potential and the rHF model: when using a single $\bm k$-point, the system has a positive gap, but smaller than when using eight $\bm k$-points and a linear potential obtained as above. The constant $c_N$ \eqref{eq:cN} is thus larger when using a single $\bm k$-points, yielding a worse upper bound. This also appears when comparing Figures~\ref{fig:Si_1kpt} and \ref{fig:Si_8kpt}: the error bound is better in the second case, for which the gap is smaller.

\begin{table}[h!]
    \centering
    \caption{Ratio between the error $E(\gamma_{N,m})-E(\gamma_\star)$ and the upper bound $\normalfont\err^{\rm SCF}_{N,m} + \err^{\rm disc}_{N,m}$ using $\etaene{0}$ and $\etaene{}$. Data displayed for both the Si crystal (with one and eight $\bm k$-points) and the HF molecule, every five steps of the SCF iterations.}
    \begin{tabular}{@{}ccccccc@{}}
        \toprule
        \multicolumn{1}{l}{\multirow{2}{*}{SCF it.}} & \multicolumn{2}{l}{Si (one $\bm k$-point)} & \multicolumn{2}{l}{Si (eight $\bm k$-points)} & \multicolumn{2}{l}{HF (eight $\bm k$-points)} \\
        \multicolumn{1}{l}{}                         & $\etaene{0}$         & $\etaene{}$         & $\etaene{0}$          & $\etaene{}$          & $\etaene{0}$          & $\etaene{}$          \\ \midrule
        1                                            & 1.53019              & 1.53019             & 1.48986               & 1.48986              & 3.13099               & 3.13099              \\
        5                                            & 4.04515              & 4.04515             & 1.22005               & 1.22010              & 2.73819               & 2.73871              \\
        10                                           & 6.80721              & 6.80721             & 1.02225               & 1.06797              & 0.99125              & 1.02102              \\
        15                                           & 4.41846              & 4.42815             & --                    & --                   & 0.98945              & 1.01923              \\
        20                                           & 3.77631              & 3.93929             & --                    & --                   & --                    & --                     \\ \bottomrule
    \end{tabular}
\label{tab:ratio_comparison}
\end{table}


\begin{figure}[p!]
    \centering
    \includegraphics[width=0.49\linewidth]{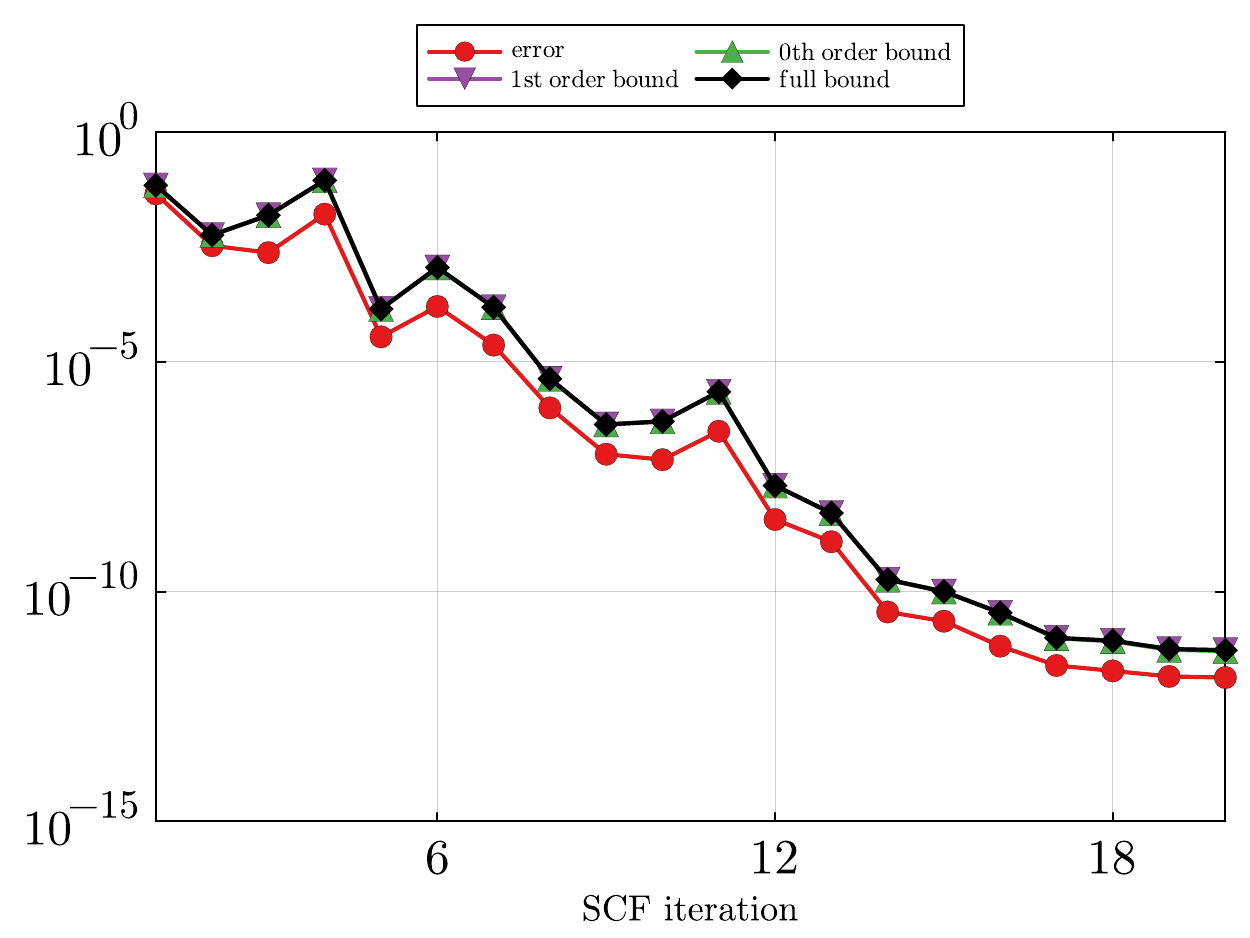}\hfill
    \includegraphics[width=0.49\linewidth]{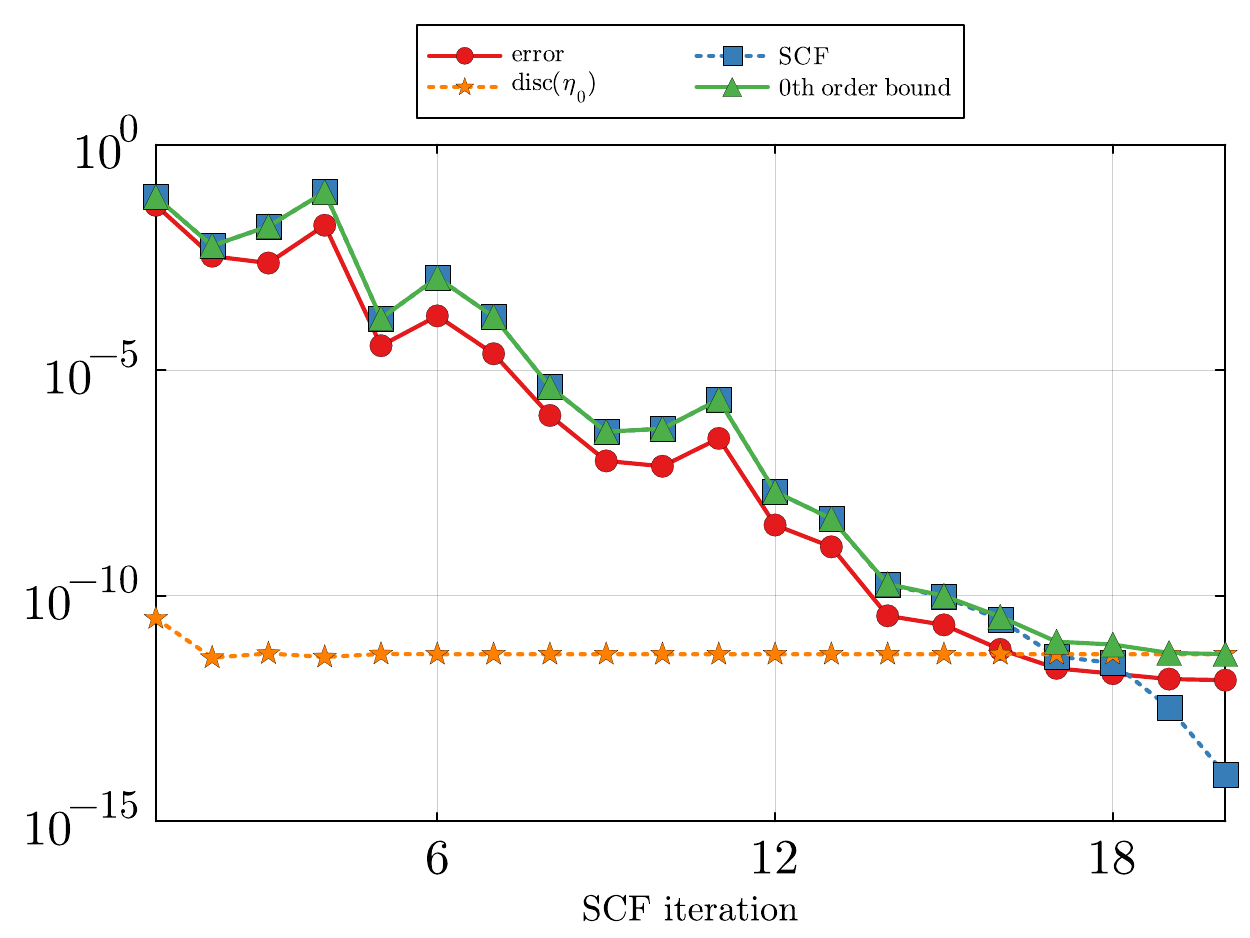}
    \caption{Tracking of the error $E(\gamma_{N,m}) - E(\gamma_\star)$ for a Si crystal (one $\bm k$-point). (Left) Full-inversion bound with zeroth- and first-order approximations. (Right) Zeroth-order bound and its splitting between SCF and discretization contributions, as in \eqref{eq: full bound}.}
    \label{fig:Si_1kpt}
\end{figure}

\begin{figure}[p!]
    \centering
    \includegraphics[width=0.49\linewidth]{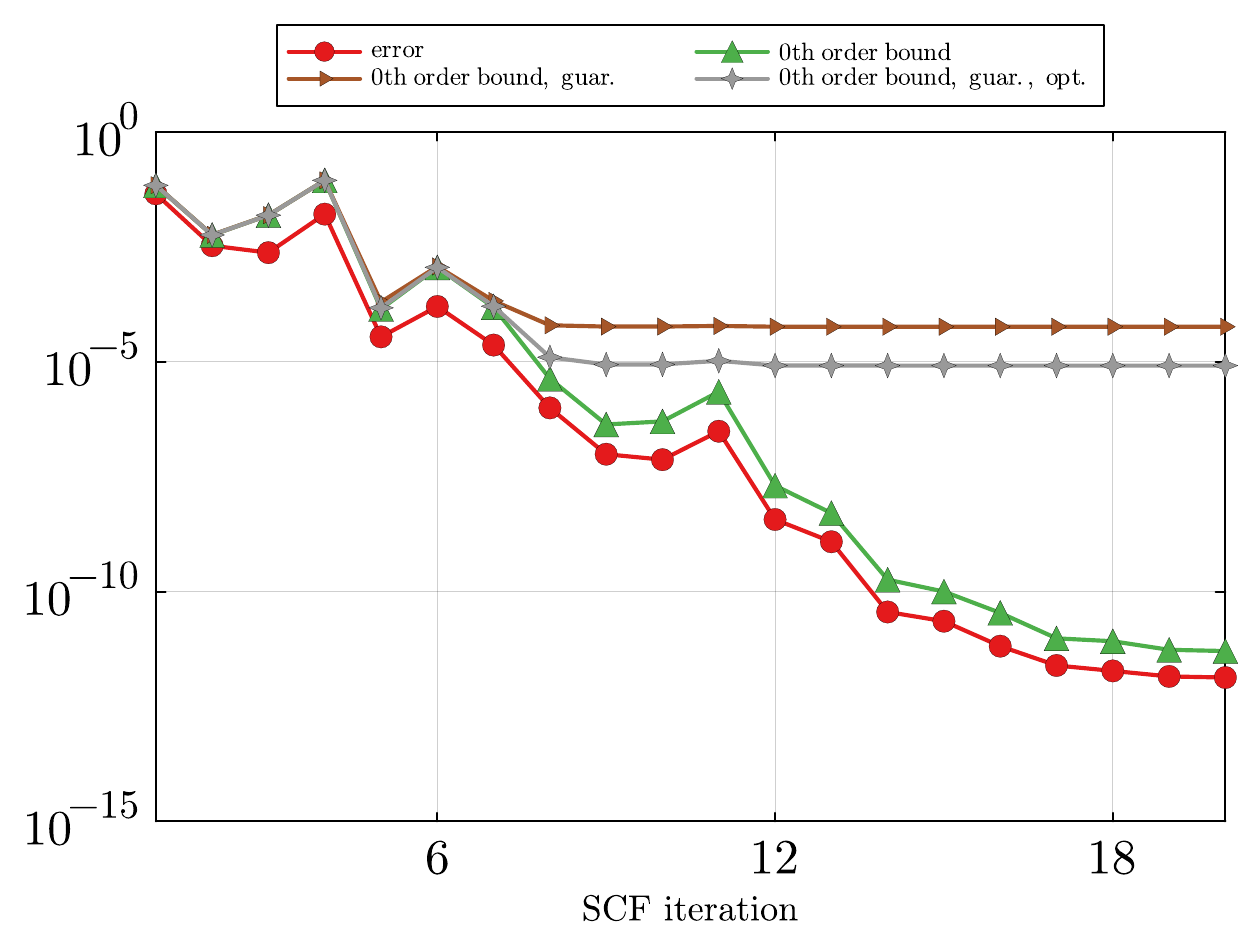}\hfill
    \includegraphics[width=0.49\linewidth]{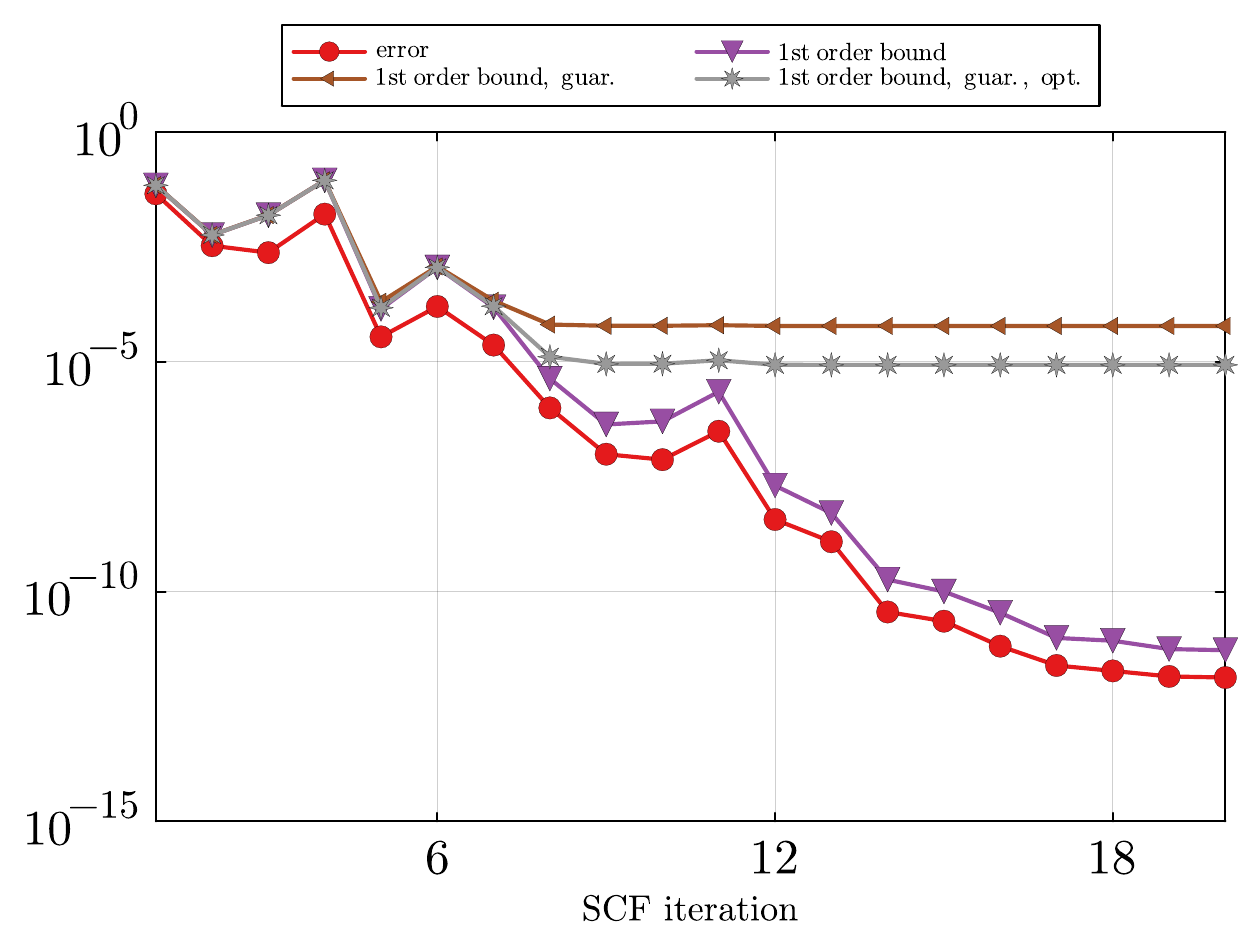}
    \caption{Tracking of the error $E(\gamma_{N,m}) - E(\gamma_\star)$ for a Si crystal (one $\bm k$-point), with the zeroth-order bound (left) and first-order bound (right). We added the associated guaranteed bounds by estimating the remainders and their optimization with respect to the shift.}
    \label{fig:Si_1kpt_optimized_shift}
\end{figure}

\begin{figure}[p!]
    \centering
    \includegraphics[width=0.49\linewidth]{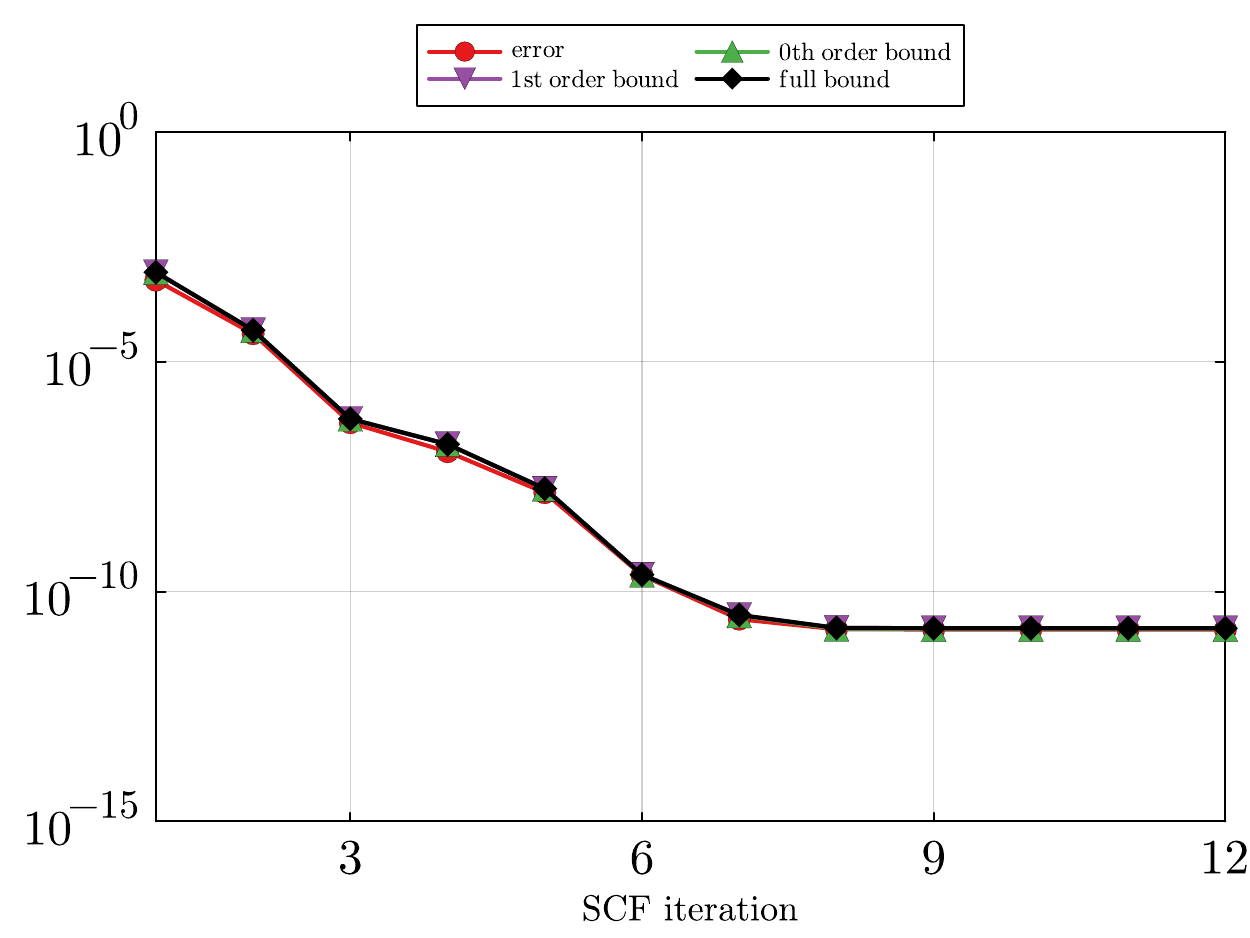}\hfill
    \includegraphics[width=0.49\linewidth]{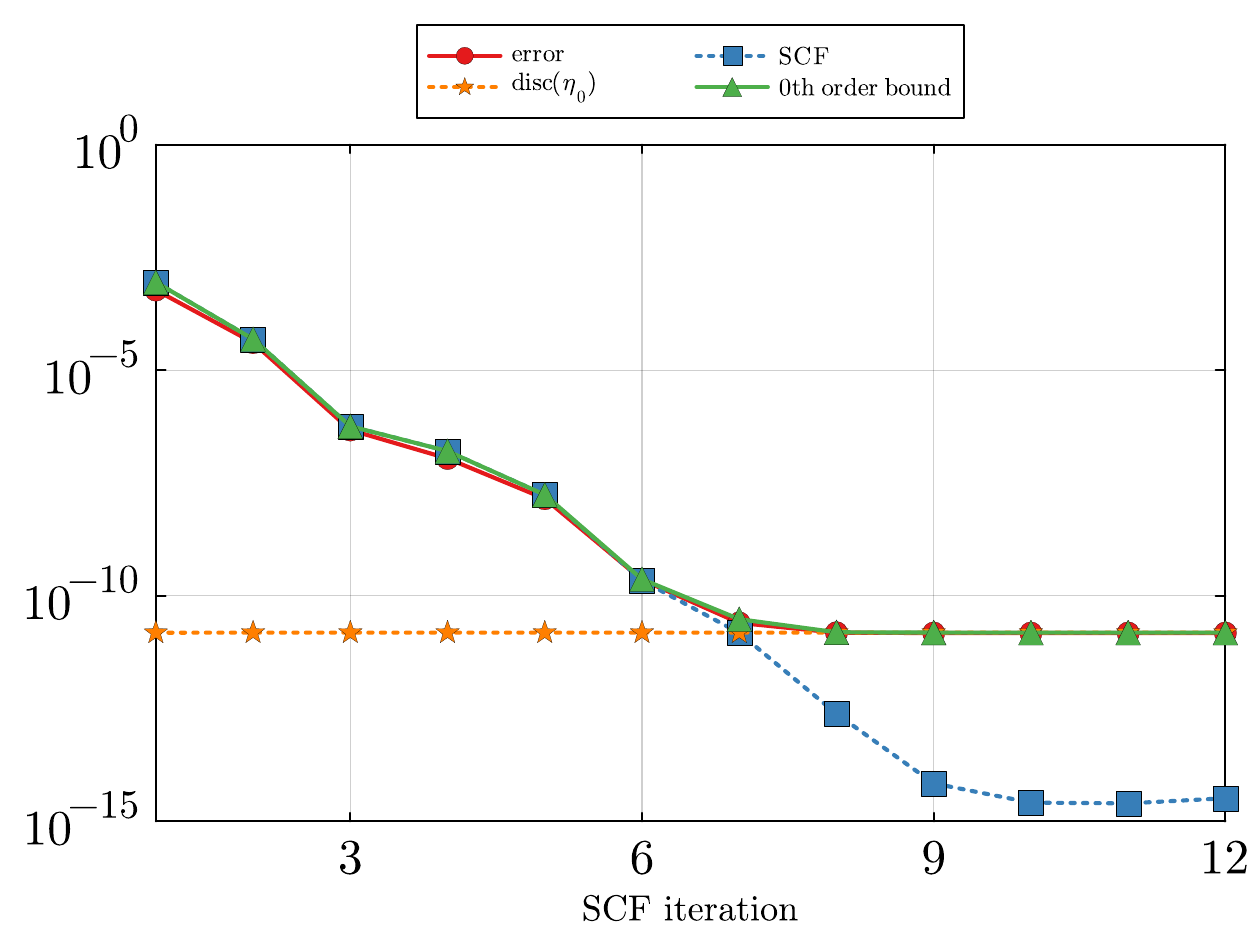}
    \caption{Tracking of the error $E(\gamma_{N,m}) - E(\gamma_\star)$ for a Si crystal (eight $\bm k$-points). (Left) Full-inversion bound with zeroth- and first-order approximations. (Right) Zeroth-order bound and its splitting between SCF and discretization contributions, as in \eqref{eq: full bound}.}
    \label{fig:Si_8kpt}
\end{figure}

\clearpage
\subsubsection{Hydrogen Fluoride (HF) molecule}
We run here a simulation for eight $\bm k$-points, with $\Ecutref=1500\ \Ha$, $\Ecut=750\ \Ha$ (we use the same trick to ensure convergence of the SCF as the one we used for the Silicon system). In Figure~\ref{fig:HF_8kpt}, we plot the results from the said simulation and in Table~\ref{tab:ratio_comparison} we display the ratio between the energy difference and the zeroth-order bound for the simulation with eight $\bm k$-points. We notice again that the zeroth-order approximation provides a very satisfying estimation of the true error, with a similar transition between the regime where the SCF error dominates and the one where the discretization error dominates. Note however that, while the full inversion is guaranteed, this time the zeroth-order bound is not, due to the truncation of the Neumann series.

\begin{figure}[h!]
    \centering
    \includegraphics[width=0.49\linewidth]{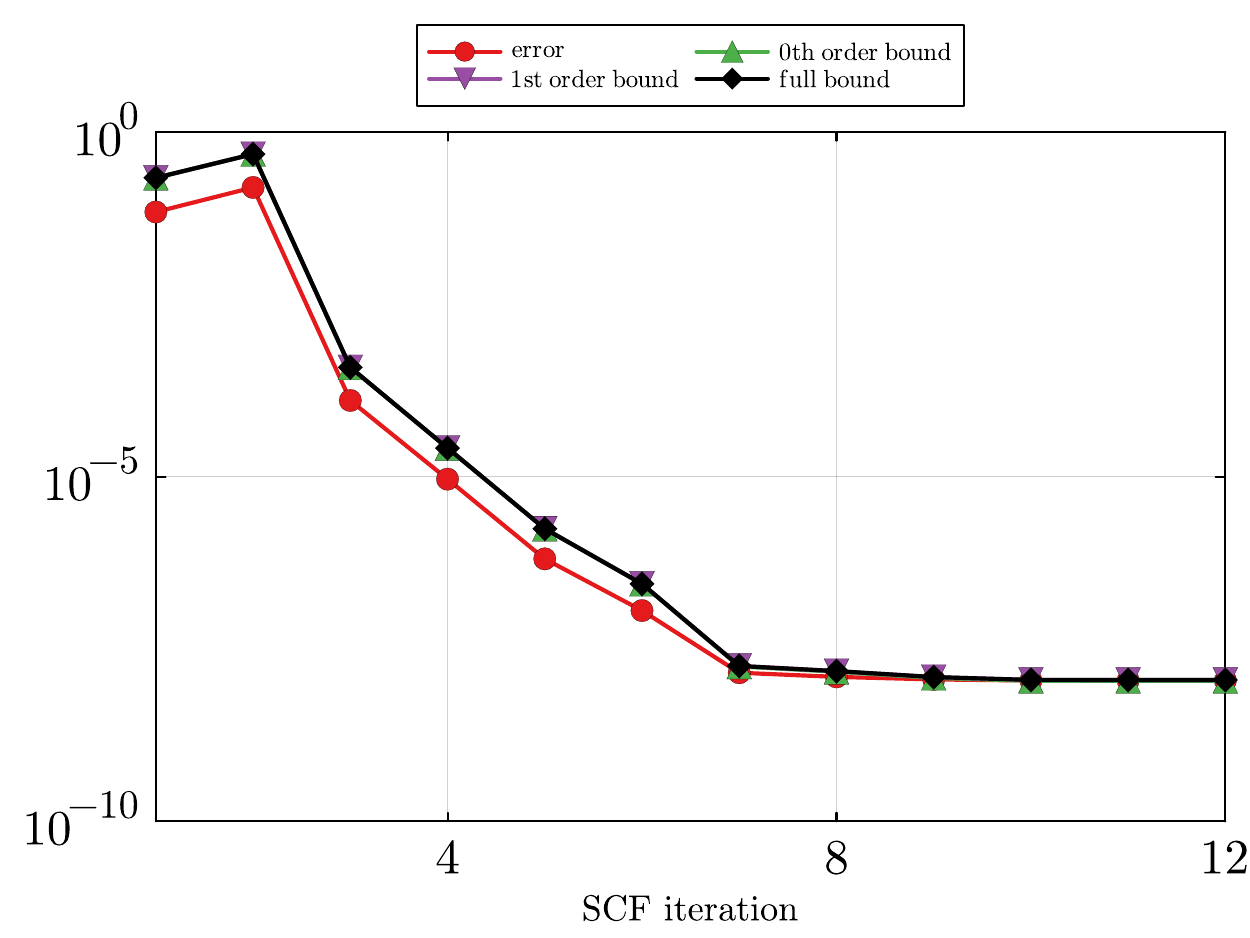}\hfill
    \includegraphics[width=0.49\linewidth]{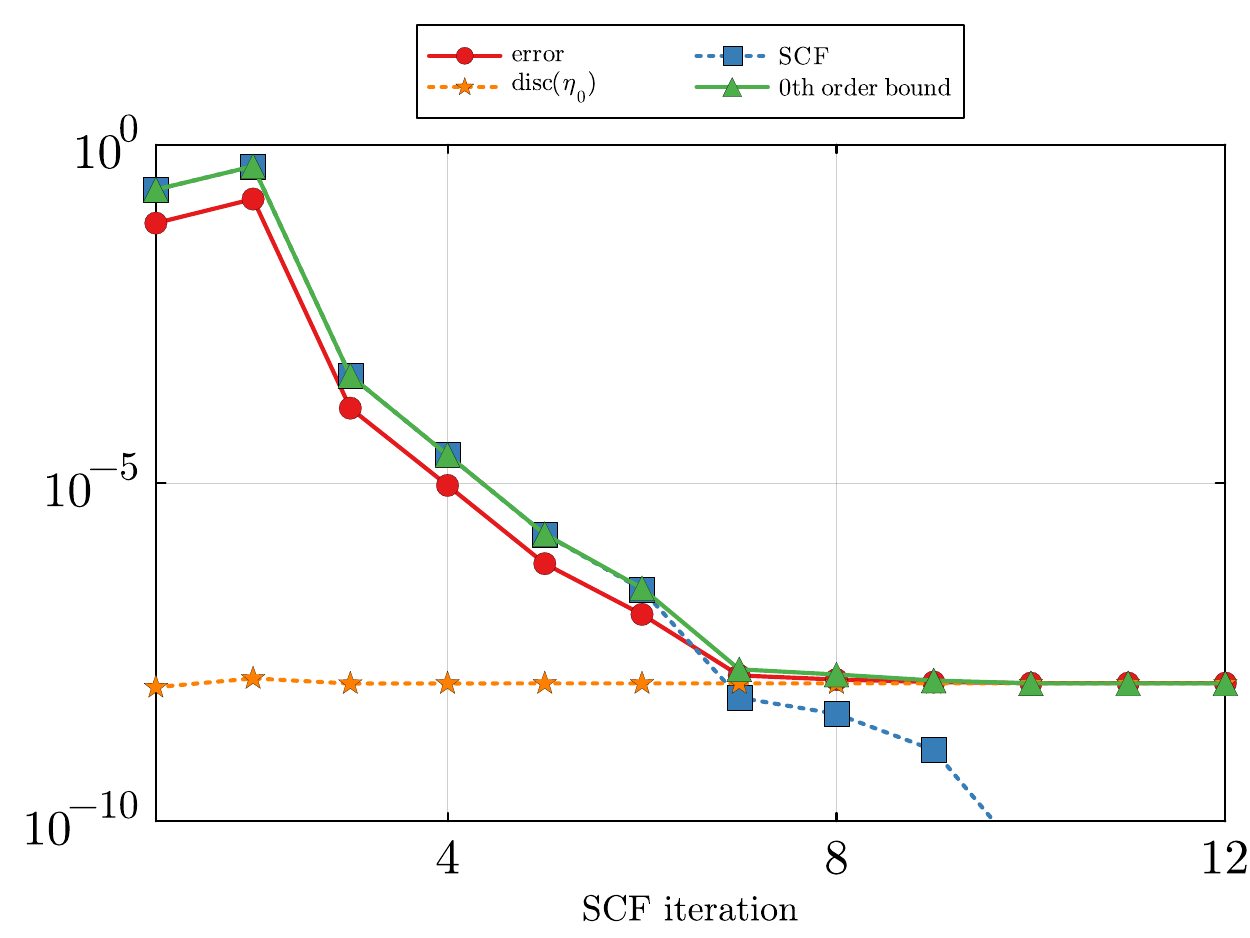}
    \caption{Tracking of the error $E(\gamma_{N,m}) - E(\gamma_\star)$ for a HF molecule (eight $\bm k$-points). (Left) Full-inversion bound with zeroth- and first-order approximations. (Right) Zeroth-order bound and its splitting between SCF and discretization contributions, as in \eqref{eq: full bound}.}
    \label{fig:HF_8kpt}
\end{figure}

\subsection{Usage of the bounds in practice and conclusion}

We would like to end this section with a brief discussion on the practical aspects of the bounds derived in this paper. Table \ref{tab:execution_time} summarizes the execution time for one of the simulations, which will be useful to decide which bound to employ.
First, the bound based on the full-inversion of the Hamiltonian is clearly the closest to the exact error.
Unfortunately, performing the full-inversion at each step is too costly to be used in practical simulations.
Next, the approximate inversions by means of a zeroth or first-order truncation seem to lead to very similar results. Hence, the zeroth-order approximation being the cheapest to compute, it seems the most relevant one. Indeed, even if it is in the finite dimensional space $\cV_N$, it is still necessary to perform a full operator inversion in this space to compute the first-order approximation: solving the underlying linear system is as costly as performing a full SCF cycle in the space $\cV_N$, which greatly reduces its applicability.
Finally, it appears that adding the estimation of the remainders from the truncation of the Neumann series worsen the bounds by more than one order of magnitude: mathematically guaranteed bounds seem to be only possible at the expanse of accuracy, a conclusion some of us also reached in a previous paper \cite{cancesPracticalErrorBounds2022} about discretization errors in the calculation of interatomic forces. These considerations therefore support, for practical applications, the choice of the zeroth-order bound, without the estimation of the remainders from the Neumann series.

\begin{remark}[Truncation of the Neumann series]
  Let us first rewrite the Neumann series in
  \eqref{eq:3-4} as
  \[
	  A^{-1}r_{i,N} =
    H_0^{-1/2}\biggr(\sum_{n=0}^{+\infty}\big(-H_{0}^{-1/2}WH_{0}^{-1/2}\big)^n
    \biggr)H_0^{-1/2}r_{i,N}.
  \]
  Since $B=-H_{0}^{-1/2}WH_{0}^{-1/2}$ is a compact, self-adjoint operator on
  $L^2_\per(\Omega)$, it admits an eigenbasis $(e_k)_{k\in\N}$, with associated
  eigenvalues $(\lambda_k)_{k\in\N}$ decreasing towards $0$ as $k\to+\infty$.
  Projecting $H_0^{-1/2}r_{i,N}$ in this eigenbasis, the zeroth order approximation error
  is nothing else than
  \begin{equation*}
    \begin{split}
      \norm{A^{-1}r_{i,N} - H_0^{-1}r_{i,N}}^2 &=
      \left\| H_0^{-1/2} \biggr(\sum_{n=1}^{+\infty}B^n\biggr)H_0^{-1/2}r_{i,N}
      \right\|^2 \\ & \leq \norm{H_0^{-1/2}}^2 \sum_{n=1}^{+\infty}\sum_{k=1}^{+\infty}
      |\lambda_k|^{2n}|\langle e_k, H_0^{-1/2}r_{i,N} \rangle|^2  \\
      &= \norm{H_0^{-1/2}}^2\sum_{k=1}^{+\infty}\frac{\lambda_k^2}{|1-\lambda_k^2|} |\langle e_k,
      H_0^{-1/2}r_{i,N} \rangle|^2.
    \end{split}
  \end{equation*}
  In practice, we suspect that the zeroth and first order approximations give
  satisfactory results because the terms of the series converge relatively fast
  to 0: in the numerical experiments we realised in the 1D case, the above series has
  a value of about $10^{-9}$.

\end{remark}

\begin{table}[]
\centering
\caption{Time spent in each bound computation with respect to a single simulation for an HF molecule (eight $\bm k$-points). Results for the bound based on the full-inversion of the Hamiltonian ($\eta$), the zeroth-order truncation ($\eta_0$) and the first-order truncation ($\eta_1$).}
\label{tab:execution_time}
\begin{tabular}{ccc}
\toprule
           & Time(hr) & (\%) \\ \hline
$\eta$     & 3.45 & 32.5 \\
$\eta_0$   & 0.11 & 1.0  \\
$\eta_1$   & 1.47 & 13.8 \\
Simulation & 10.6 &      \\ \bottomrule
\end{tabular}
\end{table}

\subsection{Application to nonconvex exchange-correlation functionals (LDA and PBE) for Silicon}

Finally, we apply the results from this paper along the SCF iterations for a Si crystal modelled with the LDA \cite{perdewAccurateSimpleAnalytic1992} and PBE \cite{perdewGeneralizedGradientApproximation1996} exchange-correlation functionals, where we set $\Ecutref=400\ \Ha$, $\Ecut=50\ \Ha$, together with eight $ \bm k$-points.
In these cases, $E_{\rm xc}$ is not a convex functional of the density anymore, so the bounds cannot be expected to be guaranteed.
However, as the additional term is expected to be of higher order in the discretization error (see for instance \cite{cancesNumericalAnalysisPlanewave2012a,dussonPostprocessingPlanewaveApproximation2021} for the simple $X\alpha$ LDA functional), we can expect the approximation to hold asymptotically. This is confirmed by the numerical experiment in Figure~\ref{fig:Si_LDA}, where we notice a very good approximation of the error on the energy along the SCF iterations using the simple zeroth-order bound. In particular, the transition from a SCF dominating error to a discretization dominating error clearly appears, paving the way for efficient adaptive strategies, even for nonconvex density functionals.

\begin{figure}[h!]
    \centering
    \includegraphics[width=0.49\linewidth]{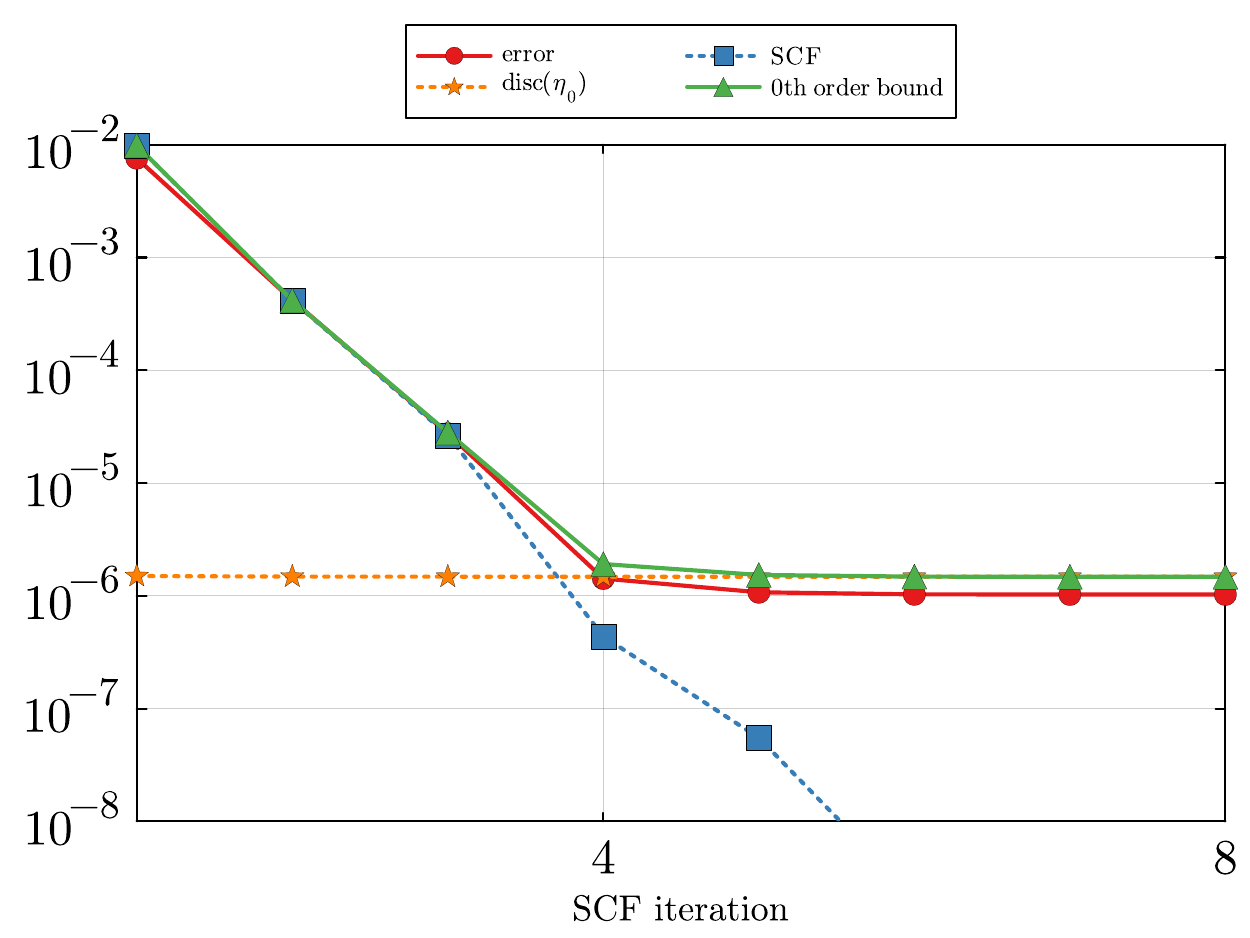}\hfill
    \includegraphics[width=0.49\linewidth]{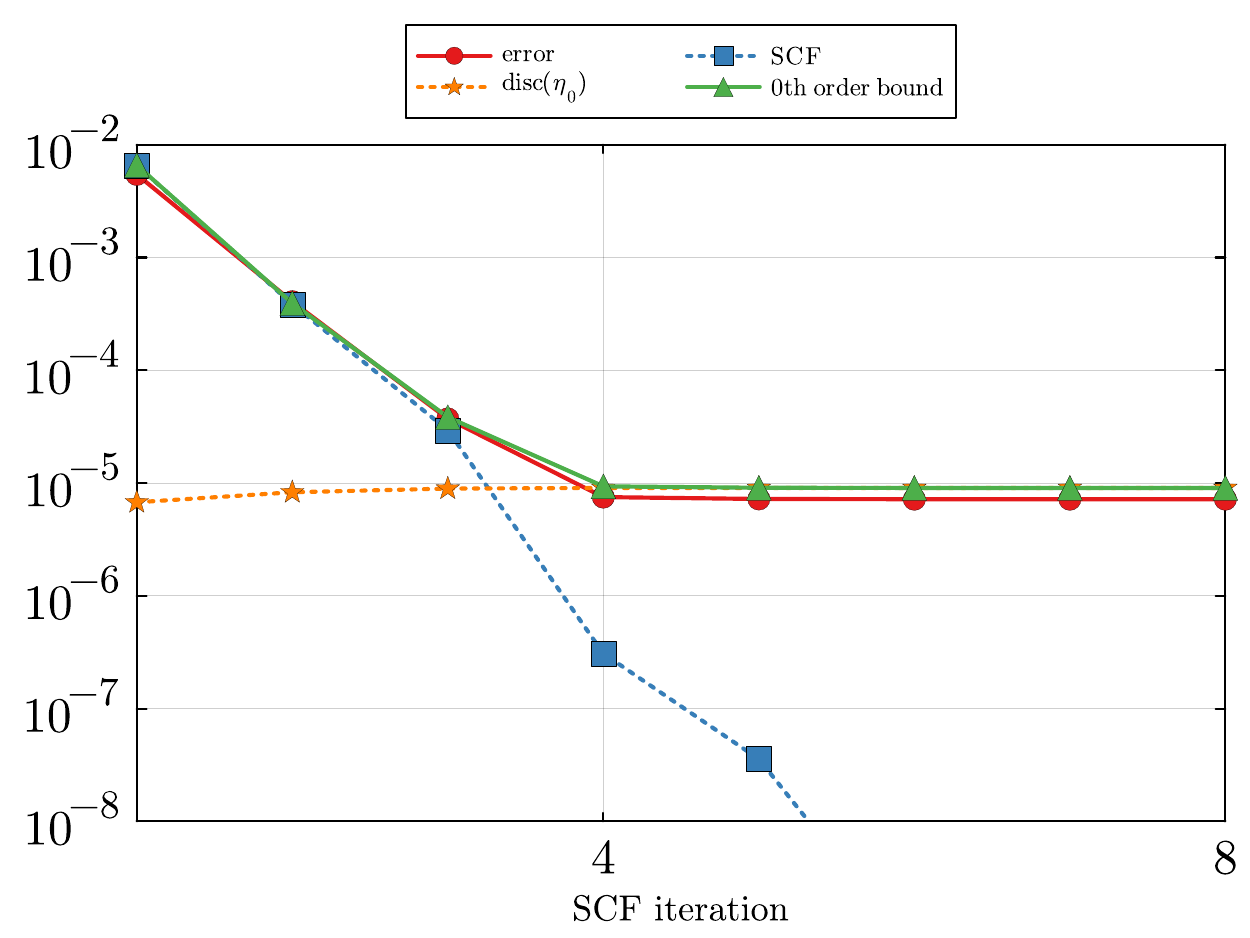}
    \caption{Tracking of the error $E(\gamma_{N,m}) - E(\gamma_\star)$ for a Si crystal with the zeroth-order bound. Data generated with eight $\bm k$-points and the LDA (left) or PBE (right) exchange-correlation functional. \enquote{\normalfont SCF} stands for the SCF error $\normalfont\err^{\rm SCF}_{N,m}$ and \enquote{$\normalfont{\rm disc}(\eta_0)$} stands for the discretization error $\normalfont\err^{\rm disc}_{N,m}$, computed with $\eta_0$ (cf. Table~\ref{tab:equation_reference_table}). The transition between a regime where the SCF error dominates to a regime where the discretization error dominates clearly appears, even though the bound is not mathematically guaranteed (nonconvex model and truncation of the Neumann series). }
    \label{fig:Si_LDA}
\end{figure}

\section*{Data availability}

All the codes used to generate the plots from this paper are available at \url{https://doi.org/10.18419/darus-4469}.

\bibliographystyle{siam}
\bibliography{refs}

\appendix

\section{Periodic potentials and Brillouin zone discretization}\label{app:brillouin_zone}

In this appendix, we explain how to extend the bounds developed in the rest of the paper to the case of Kohn--Sham equations with periodic potentials and Brillouin zone discretization, as explained in Remark~\ref{rmk:brillouin}. Note that, in practice, a finite set of $\bm k$-points is used (typically, a uniform Monkhorst--Pack grid \cite{monkhorstSpecialPointsBrillouinzone1976}), introducing a numerical quadrature error when integrating over the Brillouin zone. Therefore, all the integrals on the Brillouin zone appearing in this appendix are replaced by quadrature rules: this discretization error is not taken into account here. The interested reader is referred for instance to \cite{cancesNumericalQuadratureBrillouin2020} for the \emph{a priori} numerical analysis of such quadrature methods.

First, let us recall some basic properties of the Bloch--Floquet transform (see for instance \cite[Section XIII.16]{reedAnalysisOperators1978} and \cite{levitt_screening_2020} for more details). For a generic operator $A$ acting on $L^2(\R^3;\C)$ which commutes with $\cR$-translations, the Bloch--Floquet transform $\cZ$ gives the decomposition
\begin{equation}
\cZ^*A\cZ = \fint_{\cB} A_{\bm k}{\rm d}\bm k,
\end{equation}
where $\fint_{\cB} = \frac1{|\cB|}\int_\cB$ is the averaged integral over the Brillouin zone $\cB$ (defined as the first Voronoï cell of the reciprocal lattice $\cR^*$), $(A_{\bm k})_{\bm k\in\cB}$ is the set of the Bloch fibers of $A$ and each $A_{\bm k}$ is an operator acting on $L^2_\per(\Omega)$. If the operator $A$ is locally trace-class, then its \emph{trace per unit cell} is given by
\begin{equation}
\trvo(A) = \fint_{\cB}\tr(A_{\bm k}){\rm d}\bm k,
\end{equation}
where here and below $\tr$ denotes the trace of operators on $\cH=L^2_\per(\Omega)$.
We now change the density matrix manifold $\cM$ defined in \eqref{eq:cM} to
\begin{multline}
    \cMvo \coloneqq \Big\{ \gamma \in \cS(L^2(\R^3)),\;\gamma^2=\gamma,\; \gamma\tau_{\bm R} = \tau_{\bm R}\gamma\ \forall\ \bm R\in\cR,\\ \trvo(\gamma)=\Ne,\; \trvo(-\Delta\gamma) < +\infty\Big\},
\end{multline}
where, for a given vector $\bm R\in\cR$, $\tau_{\bm R}$ is the translation operator on $L^2(\R^3;\C)$ defined by
\[
\forall\ \phi\in L^2(\R^3;\C),\ (\tau_{\bm R}\varphi)(\bm x) = \varphi(\bm x - \bm R) \text{ for a.e. } \bm x\in\R^3,
\]
and where $\Ne$ is the number of electrons \emph{per unit cell}. The conditions $\;\gamma^2=\gamma$, $\gamma\tau_{\bm R} = \tau_{\bm R}\gamma$ for all $\bm R\in\cR$, and $\trvo(\gamma)=\Ne$, imply that the Bloch decomposition of $\gamma$ is given by
\begin{equation}
\cZ^*\gamma\cZ = \fint_{\cB}\gamma_{\bm k}{\rm d}\bm k \quad \mbox{with} \quad  \gamma_{\bm k}= \sum_{i=1}^{N_{\bm k}}
|\phi_{i,\bm k}\rangle\langle \phi_{i,\bm k}|,
\end{equation}
where the function $\cB \ni \bm k \mapsto N_{\rm k} \in \N$ is integrable and such that
\[
    \fint_\cB N_{\bm k}{\rm d} \bm k = \Ne,
\]
and where $( \phi_{i,\bm k})_{1 \le i \le N_{\bm k}}$ forms an orthonormal family of $L^2_\per(\Omega)$. The term $\trvo(-\Delta\gamma)$ is defined in the same spirit of \eqref{def: weak trace}. As
\[
 \cZ^* \left( -\Delta\right) \cZ =   \fint_{\cB} (-\i\nabla + \bm k)^2{\rm d}\bm k ,
\]
we can generalize the notion of trace as in  \eqref{def: weak trace} and obtain
$$
\trvo(-\Delta\gamma):= \fint_{\cB} \tr\left((-\i\nabla + \bm k)^2 \gamma_{\bm k}\right) {\rm d}\bm k
= \fint_{\cB} \sum_{i=1}^{N_{\bm k}} \| (-\i\nabla + \bm k) \phi_{i,\bm k}\|_{L^2_\per(\Omega)}^2  {\rm d}\bm k.
$$

Next, assuming that the potential $V$ is in $L^2_{\rm per}(\Omega;\R)$, we can consider $h = -\frac12\Delta + V$ as a self-adjoint operator acting on $L^2(\R^3;\C)$, with domain $H^2(\R^3;\C)$ and form domain $H^1(\R^3;\C)$, that commute with every $\cR$-translations and is therefore decomposed by the Bloch transform
\begin{equation}
\cZ^*h\cZ = \fint_{\cB}h_{\bm k}{\rm d}\bm k,
\end{equation}
where $h_{\bm k} = \frac12(-\i\nabla + \bm k)^2 + V$ is an operator on $L^2_\per(\Omega)$ with domain $H^2_\per(\Omega)$ and form domain $H^2_\per(\Omega)$.
Using the generalized trace notation \eqref{def: weak trace}, we define
\begin{equation}
    \trvo(h\gamma) \coloneqq \fint_{\cB} \tr(h_{\bm k}\gamma_{\bm k}) {\rm d}\bm k,
\end{equation}
and we can write the \emph{energy per unit cell} as
\begin{equation}
    \Evo(\gamma) \coloneqq \trvo(h\gamma) + F(\rho_\gamma),
\end{equation}
where the density $\rho_\gamma$ is computed as
\begin{equation}
\rho_\gamma(\bm x) = \fint_{\cB} \gamma_{\bm k}(\bm x,\bm x){\rm d}\bm k =
\fint_{\cB} \sum_{i=1}^{N_{\bm k}} |\phi_{i,\bm k}(\bm x)|^2{\rm d}\bm k.
\end{equation}
Note that, by arguments similar to those used in Section~\ref{sec:dm}, $\rho_\gamma$ belongs to $L^2_\per(\Omega;\R)$.
The nonlinear term $F(\rho_\gamma)$ is then defined for instance by \eqref{eq:KS_nonlin} in Kohn--Sham DFT.

With such a framework, one can easily extend the main results to the Kohn--Sham equations with periodic potentials and Brillouin zone discretization. Indeed, we have
\begin{equation}
    \begin{split}
        \forall\ \gamma_1,\gamma_2\in\cMvo,\ \langle F'(\rho_{\gamma_1}), \rho_{\gamma_2}\rangle & = \int_\Omega V_{\rho_{\gamma_1}} \rho_{\gamma_2}  = \fint_\cB \tr(V_{\rho_{\gamma_1}}\gamma_{2,\bm k}){\rm d}\bm k = \trvo(V_{\rho_{\gamma_1}}\gamma_2).
    \end{split}
\end{equation}
The Bloch fibers of the Kohn--Sham Hamiltonian are then given by $H_{\rho,\bm k} = h_{\bm k} + V_{\rho} = \frac12(-\i\nabla + \bm k)^2 + V + V_{\rho}$, and have domain $H^2_\per(\Omega)$ and form domain $H^1_\per(\Omega)$. Note that the Kohn--Sham equations with periodic potentials \eqref{eq:KS-DFT-per} are obtained as the Euler--Lagrange equations of the minimisation problem
\begin{equation}\label{eq:min_pb_bz}
\min_{\gamma\in\cMvo} \Evo(\gamma).
\end{equation}
Existence and uniqueness of a minimizer to this problem is studied for instance in \cite[Theorem 1]{cances_new_2008} for the rHF model. A crucial assumption in our approach is that the band gap $\nu$ is positive, where $\nu = \min_{\bm k\in\cB} \varepsilon_{N_{\bm k}+1,\bm k} - \max_{\bm k\in\cB} \varepsilon_{N_{\bm k},\bm k}$ with $(\varepsilon_{i,\bm k})_{i\in\N}$ the eigenvalues of the Bloch fiber $H_{\rho_{\gamma},\bm k}$. We make this assumption in the sequel. From a physical point of view, this means that the system under consideration is an \emph{insulator}. In this case, it holds $N_{\bm k} = \Ne$ for a.e. ${\bm k}\in\cB$.

Assuming that $F$ is convex, everything then follows similarly to the computations from Section~\ref{sec:3} by (i) replacing the $L^2_\per(\Omega)$ trace $\tr$ with the trace per unit cell $\trvo$ and (ii) replacing the real number $\mu\in\R$ with a function $\mu: \cB \ni \bm k \mapsto \mu_{\bm k} \in \R$. Then, for fixed $\gamma_1,\gamma_2\in\cMvo$, if we are able to find a function $\mu$ such that,
\begin{equation}\label{eq:mu_brillouin}
    \trvo\big((h + V_{\rho_{\gamma_2}} - \mu)\gamma_1\big)\geq 0,
\end{equation}
then the extension of Corollary~\ref{cor:bound} to the energy per unit cell reads
\begin{equation}
   \Evo(\gamma_2) - \Evo(\gamma_1) \leq \trvo\big((h + V_{\rho_{\gamma_2}} - \mu)\gamma_2\big).
\end{equation}
In order for \eqref{eq:mu_brillouin} to hold, since
\begin{equation}
    \trvo\big((h + V_{\rho_{\gamma_2}} - \mu)\gamma_1\big) = \fint_\cB \tr\big((h_{\bm k} + V_{\rho_{\gamma_2}} - \mu_{\bm k}) \gamma_{1,\bm k} \big){\rm d}\bm k,
\end{equation}
it is sufficient to compute $(\mu_{\bm k})_{\bm k\in\cB}$ such that
\begin{equation}
    \forall\ \bm k\in\cB,\ \tr\big((h_{\bm k} + V_{\rho_{\gamma_2}} - \mu_{\bm k})\gamma_{1,\bm k}\big) \geq 0.
\end{equation}
Therefore, one can apply the strategy presented in Section~\ref{sec:guaranteed} to each Bloch fibers $h_{\bm k} + V_{\rho_{\gamma_2}}$ to compute an admissible function $\mu$. This ultimately amounts to compute (or approximate) an $\eta_{\bm k}$ from \eqref{eq:eta_energy} for every $\bm k$ in $\cB$.
Everything we presented to compute this quantity can therefore be applied for every $\bm k$ in $\cB$, or in practice for every point of the grid used to discretize the integrals over $\cB$.

From a practical point of view, the Kohn--Sham equation with Brillouin zone discretization \eqref{eq:KS-DFT-per} are also solved with SCF algorithms, except that each Bloch fiber $H_{\rho,\bm k}$ has to be diagonalized and the density rebuilt by integrating over the Brillouin zone. A similar splitting between discretization error and SCF error, as in \eqref{eq: full bound}, can thus also be derived: the guaranteed bound on the energy per unit cell then reads, at iteration $m$ of the SCF in the space $\cV_N$,
\begin{equation}
    	\boxed{\Evo(\gamma_{N,m}) - \Evo(\gamma_\star) \le\err_{N,m}^{\rm disc} + \err_{N,m}^{\rm SCF},}
\end{equation}
where each error component has to be computed fiber-wise. In other words,
\begin{equation}
    \begin{split}
        \err_{N,m}^{\rm disc} &= \trvo\big((H_{\rho_{\gamma_{N,m}}} - \mu^{\rm lb}_{N,m+1})\gamma_{N,m+1} \big) \\ &=
        \fint_{\cB} \tr\big((H_{\rho_{\gamma_{N,m}},\bm k} - \mu^{\rm lb}_{N,m+1,\bm k})\gamma_{N,m+1,\bm k} \big){\rm d}\bm k, \\
        \err_{N,m}^{\rm SCF}  &= \trvo\big(H_{\rho_{\gamma_{N,m}}}\gamma_{N,m}\big) - \trvo\big(H_{\rho_{\gamma_{N,m}}}\gamma_{N,m+1}\big) \\ & = \fint_\cB \tr\big(H_{\rho_{\gamma_{N,m}},\bm k}\gamma_{N,m,\bm k}\big) - \tr\big(H_{\rho_{\gamma_{N,m}},\bm k}\gamma_{N,m+1,\bm k}\big){\rm d}\bm k.
    \end{split}
\end{equation}

\section{Neumann series truncation error}\label{app:opnorm}

In this section we provide a way to estimate  \(\|H_{0}^{-1}W\|\). As before we make use of the decomposition $\cH=\cV_N\oplus \cV_N^\perp$, and write
\begin{equation}
	\|H_{0}^{-1}W\| \leq \|H_{0}^{-1}\Pi_N W\| + \|H_{0}^{-1}\Pi_N^\perp W\|.
\end{equation}
The second term can be easily estimated as follows:
\begin{equation}
	\begin{split}
		\|H_{0}^{-1}\Pi_N^\perp W\| \leq &
		 \left\|\left(-\frac{1}{2}\Delta + \langle V\rangle\right)^{-1}\Pi_N^\perp (V-\langle V \rangle)\Pi_N^\perp \right\| \\
      + & \left\|\left(-\frac{1}{2}\Delta + \langle V\rangle\right)^{-1}\Pi_N^\perp V \Pi_N \right\| \\
      \leq & \left\|\left(-\frac{1}{2}\Delta + \langle V\rangle\right)^{-1}\Pi_N^\perp \right\| \Big(\left\|(V-\langle V \rangle)\Pi_N^\perp \right\| + \| V \Pi_N \| \Big) \\ \leq &
		\frac{1}{\Ec+\langle V\rangle} \Big(\|V-\langle V \rangle\|_{L^{\infty}} + \|V\|_{L^\infty}\Big).
	\end{split}
\end{equation}
The final inequality is a consequence of the Fourier representation of the operator \(-\frac{1}{2}\Delta + \langle V\rangle\) and the definition \eqref{eq:disc_space} of the subspace \(\cV_N\).

We now proceed to estimate the first term. To this end let us denote by \(\pi_\Ne\) the orthogonal projection on the vector space generated by the approximate eigenvectors \(\cW_\Ne={\rm Span}(\varphi_{i,N})_{i=1,\cdots,\Ne}\subset \cV_N\)
and \(\pi_\Ne^\perp\) as the corresponding orthogonal projection in \(\cV_N\). Introducing $A_N = \Pi_NA\Pi_N = \Pi_N H_0 \Pi_N$, we can estimate $\|H_{0}^{-1}\Pi_N W\|$ by
\begin{equation}
	\begin{split}
		\|H_{0}^{-1}\Pi_N W\| =&
		\|A_N^{-1}\Pi_N V\Pi_N^\perp \| \\ = &
		\|\left( \pi_\Ne A_N^{-1}\pi_\Ne  + \pi_\Ne^\perp A_N^{-1} \pi_\Ne^\perp \right)\Pi_N V\Pi_N^\perp \| \\ \leq &
		\| \pi_\Ne A_N^{-1}\pi_\Ne \Pi_N V\Pi_N^\perp\| + \| \pi_\Ne^\perp A_N^{-1} \pi_\Ne^\perp \Pi_N V\Pi_N^\perp \| \\ \leq &
		\| \pi_\Ne A_N^{-1}\pi_\Ne \Pi_NV\Pi_N^\perp\| + \| \pi_\Ne^\perp A_N^{-1} \pi_\Ne^\perp\| \| \Pi_N V\Pi_N^\perp \| \\ \leq &
		\| \pi_\Ne A_N^{-1}\pi_\Ne \Pi_NV\Pi_N^\perp\| +  \frac{\|V\|_{L^\infty}}{\varepsilon_{\Ne,N}}.
	\end{split}
\end{equation}
The first term can be further computed in terms of the residuals from the eigendecomposition \eqref{eq:eig_pb_lin_disc}. Indeed, since $\Pi_N^\perp \varphi_{i,N} = 0$, we have $\Pi_N^\perp r_{i,N} = \Pi_N^\perp A \varphi_{i,N}$ from which we can write, using in addition that $\Pi_N\Delta\Pi_N^\perp = 0$,
\begin{equation}\label{eq:opnorm}
	\begin{split}
		\| \pi_\Ne A_N^{-1}\pi_\Ne\Pi_N V\Pi_N^\perp\| = &
		\left\|\sum_{i=1}^\Ne\frac{|\varphi_{i,N}\rangle\langle \varphi_{i,N}|}{\varepsilon_{i,N}}\Pi_N V\Pi_N^\perp\right\| \\
        = &\left\|\sum_{i=1}^\Ne\frac{|\varphi_{i,N}\rangle\langle \varphi_{i,N}|}{\varepsilon_{i,N}}\Pi_N\left(-\frac12\Delta+V\right)\Pi_N^\perp\right\| \\
        = & \left\|\sum_{i=1}^\Ne\frac{|\varphi_{i,N}\rangle\langle \varphi_{i,N}|}{\varepsilon_{i,N}}\Pi_N A \Pi_N^\perp\right\|\\
        =& \left\|\sum_{i=1}^\Ne\frac{|\varphi_{i,N}\rangle\langle r_{i,N}|}{\varepsilon_{i,N}}\right\|.
	\end{split}
\end{equation}
Since the family $(\varphi_{i,N})_{i=1,\dots,\Ne}$ is $L^2$-orthonormal, the matrix
\[
    X = \sum_{i=1}^\Ne\frac{|\varphi_{i,N}\rangle\langle r_{i,N}|}{\varepsilon_{i,N}} \in \R^{N\times N}
\]
satisfies
\[
 X^* X = \sum_{i=1}^\Ne\frac{|r_{i,N}\rangle\langle r_{i,N}|}{(\varepsilon_{i,N})^2}  \in \R^{N\times N}.
\]
The operator norm $\|X\|$ in \eqref{eq:opnorm} is thus given by $\sqrt{\rho(X^*X)}$, where $\rho$ denotes the spectral radius. Finally, this norm is actually computable for a negligible computational cost (even for a large basis size $N$) as the square root of the highest eigenvalue of the smaller matrix $R^*R\in\R^{\Ne\times\Ne}$ with
\[
     \R^{N\times\Ne} \ni R = \begin{bmatrix}
        r_{1,N} & r_{2,N} & \cdots & r_{\Ne,N}
    \end{bmatrix} \times \begin{bmatrix}
        1/\varepsilon_{1,N}   &                     &        & \\
                              & 1/\varepsilon_{2,N} &        & \\
                              &                     & \ddots & \\
                              &                     &        & 1/\varepsilon_{\Ne,N}
    \end{bmatrix},
\] where the matrix on the right is diagonal.

\end{document}